\documentclass[12pt,letterpaper]{article}

\pdfoutput=1

\usepackage{graphics,epsfig,graphicx,color}
\graphicspath{{/EPSF/}{../Figures/}{Figures/}}
\usepackage{subfigure}

\usepackage{amssymb,latexsym}

\usepackage{setspace}

\usepackage{amsmath}

\usepackage{mathrsfs,amsfonts}

\usepackage{amsthm,amsxtra}

\usepackage[final]{showkeys}

\usepackage{stmaryrd}

\usepackage{algorithm}
\usepackage{algorithmicx}
\usepackage{algpseudocode}

\usepackage[openbib]{currvita}

\newif\ifPDF
\ifx\pdfoutput\undefined
	\PDFfalse
\else
	\ifnum\pdfoutput > 0
		\PDFtrue
	\else
		\PDFfalse
	\fi
\fi

\ifPDF
	\usepackage{pdftricks}
	\begin{psinputs}
		\usepackage{pstricks}
		\usepackage{pstcol}
		\usepackage{pst-plot}
		\usepackage{pst-tree}
		\usepackage{pst-eps}
		\usepackage{multido}
		\usepackage{pst-node}
		\usepackage{pst-eps}
	\end{psinputs}
\else
	\usepackage{pstricks}
\fi

\ifPDF
	\usepackage[debug,pdftex,colorlinks=true, 
	linkcolor=blue, bookmarksopen=false,
	plainpages=false,pdfpagelabels]{hyperref}
\else
	\usepackage[dvips]{hyperref}
\fi

\usepackage{fancyhdr}

\newtheorem{theorem}{Theorem}[section]
\newtheorem{lemma}[theorem]{Lemma}

\newtheorem{proposition}[theorem]{Proposition} 
\newtheorem{remark}[theorem]{Remark} 
\newtheorem{corollary}[theorem]{Corollary}

\newcommand{\dsum}{\displaystyle\sum}
\newcommand{\dint}{\displaystyle\int}
\newcommand{\eps}{\varepsilon}

\newcommand{\sfc}{\mathsf c}
\newcommand{\sfC}{\mathsf C}

\newcommand{\bbR}{\mathbb R} \newcommand{\bbS}{\mathbb S}

\newcommand{\bPi}{\boldsymbol \Pi}

 \newcommand{\bn}{\mathbf n}

 \newcommand{\bv}{\mathbf v} 
 \newcommand{\bx}{\mathbf x} 
 \newcommand{\bz}{\mathbf z}

\newcommand{\bI}{\mathbf I}

\newcommand{\cA}{\mathcal A}

 \newcommand{\cH}{\mathcal H}
\newcommand{\cI}{\mathcal I} 
\newcommand{\cK}{\mathcal K} 
 \newcommand{\cN}{\mathcal N}

\newenvironment{keywords}
{\noindent{\bf Key words.}\small}{\par\vspace{1ex}}

\title{Inverse transport problems in quantitative PAT for molecular imaging}

\author{
	Kui Ren\thanks{
		Department of Mathematics and ICES,
		University of Texas, 
		Austin, TX 78712; ren@math.utexas.edu .
	}
	\and 
	Rongting Zhang\thanks{
		Department of Mathematics, University of Texas, Austin, TX 78712; rzhang@math.utexas.edu .
	}
	\and 
	Yimin Zhong\thanks{
		Department of Mathematics, University of Texas, Austin, TX 78712; yzhong@math.utexas.edu .
	}
}
    
\date{}

\begin{document}

\maketitle



\begin{abstract}
Fluorescence photoacoustic tomography (fPAT) is a molecular imaging modality that combines photoacoustic tomography (PAT) with fluorescence imaging to obtain high-resolution imaging of fluorescence distributions inside heterogeneous media. The objective of this work is to study inverse problems in the quantitative step of fPAT where we intend to reconstruct physical coefficients in a coupled system of radiative transport equations using internal data recovered from ultrasound measurements. We derive uniqueness and stability results on the inverse problems and develop some efficient algorithms for image reconstructions. Numerical simulations based on synthetic data are presented to validate the theoretical analysis. The results we present here complement these in~\cite{ReZh-SIAM13} on the same problem but in the diffusive regime.
\end{abstract}






\section{Introduction}
\label{SEC:Intro}

Photoacoustic tomography (PAT)~\cite{Bal-IO12,Beard-IF11,CoLaBe-SPIE09,Kuchment-MLLE12,KuKu-HMMI10,LiWa-PMB09,PaSc-IP07,Scherzer-Book10,Wang-Book09,Wang-Scholarpedia14,Wang-PIER14} is a recent hybrid imaging modality that attempts to reconstruct high-resolution images of optical properties of heterogeneous media. In a PAT experiment, we send a short pulse of near-infra-red (NIR) photons into an optically heterogeneous medium. The photons travel inside the medium following a radiative transport process. The medium absorbs a portion of the photons during their propagation process. The absorbed photons lead to the heating of the medium which then results in a local temperature rise. The medium expanses due to
the temperature rise and then contracts when the rest of the photons leave the medium and the temperature drops accordingly. The expansion and contraction of the medium induces pressure changes which then propagate in the form of ultrasound waves. We then measure the ultrasound signals on the surface of the medium and from these measurements we intend to infer as much knowledge as possible on the optical properties, for instance the optical absorption and scattering coefficients, of the medium.

In recent years, there are great interests in developing PAT for biomedical molecular imaging~\cite{BuGrSo-NP09,RaDiViMaPeKoNt-NP09,RaNt-MP07,WaZhBaMoJi-MP12,WaMaKiHuWa-IEEE10,Wang-PIER14,WiWaWi-JNM13}. The main objective here is to visualize particular cellular functions and molecular processes inside biological tissues by using target-specific exogenous contrasts. To be specific, we consider in this work quantitative PAT for fluorescence optical imaging where one aims to image distribution of fluorescent biochemical markers inside heterogeneous media. In a typical imaging process, we first inject fluorescent markers into the medium to be probed. The markers will travel inside the medium and accumulate on their targets, for instance cancerous tissues inside the normal tissue. We then send a short pulse of NIR photons at wavelength $\lambda_x$ to the medium to excite the fluorescent markers who then emit NIR photons at a different wavelength $\lambda_m$. The absorption of both the excitation and the emission photons by the medium will then generate ultrasound waves inside the medium following the photoacoustic effect just as in a regular PAT process, assuming that fluorescence takes place instantaneously as excitation light pulse is absorbed~\cite{ReZh-SIAM13}. We then measure the ultrasound signals on the surface of the medium and attempt to recover information associated with the biochemical markers.

The density distributions for the external light source and the fluorescent light in the tissues are both described by the radiative transport equation. Let $\Omega\subset\bbR^d$ $(d\ge 2)$ be the domain of interests and $\bbS^{d-1}$ be the unit sphere in $\bbR^d$. We denote by $X=\Omega\times\bbS^{d-1}$ the phase space and $\Gamma_\pm=\{(\bx,\bv)\in\partial\Omega\times\bbS^{d-1}|\pm\bn(\bx)\cdot\bv> 0\}$ its boundary sets. We denote by $u_x(\bx,\bv)$ and $u_m(\bx,\bv)$ the density of photons at the excitation and emission wavelengths respectively, at location $\bx$, traveling in direction $\bv\in\bbS^{d-1}$. Then $u_x(\bx,\bv)$ and $u_m(\bx,\bv)$ solve the following coupled system of radiative transport equations
\begin{equation}\label{EQ:ERT QfPAT}
	\begin{array}{rcll}
	\bv\cdot\nabla u_x+(\sigma_{a,x}+\sigma_{s,x})u_x&=& \sigma_{s,x} K_\Theta(u_x),
  	&\mbox{in}\ \ X\\
  	\bv\cdot\nabla u_m+(\sigma_{a,m}+\sigma_{s,m})u_m&=& \sigma_{s,m} K_\Theta(u_m)+\eta\sigma_{a,xf}(\bx)K_I(u_x)(\bx),
  	&\mbox{in}\ \ X\\
    u_x(\bx,\bv) = g_x(\bx,\bv),&&  
    u_m(\bx,\bv) = 0, & \mbox{on}\ \ \Gamma_-
	\end{array}	
\end{equation}
where the subscripts $x$ and $m$ denote the quantities at the excitation and the emission wavelengths, respectively. The coefficients $\sigma_{a,x}$ and $\sigma_{s,x}$ (resp. $\sigma_{a,m}$ and $\sigma_{s,m}$) are respectively the absorption and scattering coefficients at wavelength $\lambda_x$ (resp. $\lambda_m$). The scattering operator $K_\Theta$ and the averaging operator $K_I$ are defined respectively as
\begin{equation}
	K_\Theta(u_x)(\bx,\bv)=\int_{\bbS^{d-1}}\Theta(\bv,\bv') u_x(\bx,\bv') d\bv',\quad\mbox{and},\quad 
	K_I(u_x)(\bx,\bv)=\int_{\bbS^{d-1}} u_x(\bx,\bv') d\bv',
\end{equation}
with the scattering kernel $\Theta(\bv,\bv')$ describing the probability that a photon traveling in direction $\bv'$ gets scattered into direction $\bv$. 

The total absorption coefficient $\sigma_{a,x}$ consists of a contribution $\sigma_{a,xi}$ from the intrinsic tissue chromophores and a contribution $\sigma_{a,xf}$ from the fluorophores of the biochemical markers: $\sigma_{a,x}=\sigma_{a,xi}+\sigma_{a,xf}$. The absorption coefficient due to fluorophores, $\sigma_{a,xf}$ is proportional to the concentration $\rho(\bx)$ and the extinction coefficient $\eps(\bx)$ of the fluorophores, i.e. $\sigma_{a,xf}=\eps(\bx)\rho(\bx)$. The coefficient $\eta(\bx)$ is the quantum efficiency of the fluorophores. The coefficients $\eta$ and $\sigma_{a,xf}$ are the main quantities associated with the biochemical markers.

The energy absorbed by the medium and the markers consists of two parts. The first part is from the excitation photons. This part can be written as $\sigma_{a,x} K_I(u_x)$.  The second part of absorbed energy comes from emission photons. This part can be written as $\sigma_{a,m} K_I(u_m)$. Therefore, the pressure field generated by the photoacoustic effect can therefore be written as:
\begin{eqnarray}
\nonumber H(\bx)&=&\Xi(\bx)\Big[\big(\sigma_{a,x}(\bx)-\eta(\bx)\sigma_{a,xf}(\bx)\big) K_I(u_x)(\bx)+\sigma_{a,m}(\bx)K_I(u_m)(\bx) \Big],\\
\label{EQ:Data QfPAT}
	&\equiv& \Xi(\bx)\Big(\sigma_{a,x}^\eta K_I(u_x)(\bx)+\sigma_{a,m}(\bx)K_I(u_m)(\bx) \Big),
\end{eqnarray}
where $\Xi$ is the (\emph{nondimensional}) Gr\"uneisen coefficient that measures the photoacoustic efficiency of the underlying medium, and $\sigma_{a,x}^\eta$ is the short notation for $\sigma_{a,xi}+(1-\eta)\sigma_{a,xf}$. We want to emphasize that when calculating the initial pressure field generated, we have subtract a portion of the energy, $\eta \sigma_{a,xf} K_I(u_x)$, from the total energy absorbed by the medium and the markers. This is because that portion of energy is used to generate fluorescence, not the heating in the photoacoustic process.

The initial pressure field  generated from the photoacoustic effect, $H$, evolves in space and time following the acoustic wave equation~\cite{BaUh-IP10,FiScSc-PRE07,StUh-IP09}:
\begin{equation}\label{EQ:Acous}
	\begin{array}{cl}
  	\dfrac{1}{c^2(\bx)}\dfrac{\partial ^2 p}{\partial t^2}-\Delta p = 0, & \mbox{in}\ \ \bbR_+\times \bbR^d\\
	p(0,\bx)= H, \quad \dfrac{\partial p}{\partial t}(0,\bx) =0, &\mbox{in}\ \ \bbR^d
	\end{array}
\end{equation}
where $c(\bx)$ is the speed of the ultrasound in the medium. The data that we measure are the solutions to the wave equation~\eqref{EQ:Acous} on the surface of the medium, $p_{|(0, t_{\rm max})\times\partial\Omega}$, $t_{\rm max}$ being large enough, for various excitation light sources.

Following~\cite{ReZh-SIAM13}, we call the process of reconstructing information on $\eta$ and $\sigma_{a,xf}$ from datum $p_{|(0, t_{max})\times\partial\Omega}$ fluorescence PAT (fPAT). This is a molecular imaging modality that combines PAT with fluorescence optical imaging. We refer interested readers to~\cite{ReZh-SIAM13} for more discussions on the mathematical modeling of fPAT, including detailed derivation and justification the models~\eqref{EQ:ERT QfPAT} (in diffusive regime) and ~\eqref{EQ:Acous}, and to~\cite{BuGrSo-NP09,RaDiViMaPeKoNt-NP09,RaNt-MP07,WaZhBaMoJi-MP12,WaMaKiHuWa-IEEE10} for some experimental and computational results on fPAT. Recent progress on fluorescence optical imaging itself can be found in~\cite{AlMeMo-OE09,AmGaGi-QAM12,GoSeEp-MP05,KuRaDuBaBo-IEEE08,PaPo-AO94,SoTaMcElNeFrAr-AO07} and references therein.

Image reconstruction in fPAT is a two-step process as in regular PAT. In the first step, we reconstruct $H$ from measured acoustic data. We assume here that this step has been finished with methods such as those in~\cite{AgKuKu-PIS09,AmBrJuWa-LNM12,BuMaHaPa-PRE07,CoArBe-IP07,FiHaRa-SIAM07,Haltmeier-SIAM11B,HaScSc-M2AS05,Hristova-IP09,KiSc-SIAM13,KuKu-EJAM08,QiStUhZh-SIAM11,StUh-IP09} and we are given the internal datum~\eqref{EQ:Data QfPAT}. Moreover, we assume that: ({\bf A-i}) the Gr\"uneisen coefficient $\Xi$ as well as the absorption and scattering coefficients of the medium at the excitation wavelength, $\sigma_{a,xi}$ and $\sigma_{s,x}$, have been known from other imaging technologies (for instance a multi-spectral quantitative PAT step~\cite{BaRe-IP12,MaRe-CMS14}) before the fluorescent biochemical markers are injected into the medium; and ({\bf A-ii}) the absorption and scattering coefficients at the emission wavelength, $\sigma_{a,m}$ and $\sigma_{s,m}$, are also reconstructed by other imaging methods (for instance a regular quantitative PAT technique~\cite{AmBoJuKa-SIAM10,BaJoJu-IP10,BaRe-IP11,BaRe-IP12,BaUh-IP10,CoArKoBe-AO06,CoTaAr-CM11,GaOsZh-LNM12,MaRe-CMS14,PuCoArKaTa-IP14,ReGaZh-SIAM13,SaTaCoAr-IP13,Zemp-AO10} after the Gr\"uneisen coefficient is known). Therefore, our main objective is only to reconstruct the quantum efficiency $\eta$ and the fluorescence absorption coefficient $\sigma_{a,xf}(\bx)$ in the system~\eqref{EQ:ERT QfPAT} from the datum $H$ in~\eqref{EQ:Data QfPAT}. This is the quantitative fPAT (QfPAT) problem.

Let us now remark on a couple of issues regarding the practical relevance of the current work. First of all, in many practical applications, it is preferable to use contrast agents that do not emit photons after absorbing incoming excitation photons. In other words, the biochemical markers have quantum efficiency $\eta=0$. In this case, the second equation in~\eqref{EQ:ERT QfPAT} drops out of the transport system, and the terms involve $\eta$ and $u_m$ all drop out from the datum~\eqref{EQ:Data QfPAT}. We are therefore back to the same mathematical problem as in a regular quantitative PAT process. The theory of the reconstruction in this case is covered in Theorem~\ref{THM:Sigma} of our results. Our results in this paper are in fact more general in the sense that we can deal with the general case of non-negligible quantum efficiency, that is $\eta > 0$. When $\eta > 0$, we have to take into account the impact of the emitted fluorescence photons in the reconstruction process. Neglecting this impact in the model would certainly introduce errors in the images reconstructed. The second issue we need to address is the difference between the work we have here and the theory on the same problem that have been developed in the diffusive regime~\cite{ReZh-SIAM13}. It is generally believed that the radiative transport equation is a more accurate model than the diffusion equation to describe the propagation of NIR photons in biological tissues~\cite{Arridge-IP99,ReBaHi-AO07}, even though it is more complicated to theoretically analyze and numerically solve. Our analysis in this paper is useful when the diffusion approximation to the radiative transport equation breaks down, for instance in media of small volumes but large mean free paths. Optical imaging of small animals~\cite{Hielscher-COB05}, for instance, is one of such biomedical applications for our work here.

The rest of the paper is organized as follows. We first present in Section~\ref{SEC:Gen} some general properties of the inverse problem, especially the continuous dependence of the datum $H$ on the unknown coefficients. We then consider in Section~\ref{SEC:One Coeff} the reconstruction of a single coefficient from a single internal data set. We derive some uniqueness and stability results on the reconstruction. In Section~\ref{SEC:Two Coeff} we study the problem of reconstructing two coefficients simultaneously, mainly in linearized settings. We then present some numerical simulations based on synthetic data in Section~\ref{SEC:Num} to validate the theory and the reconstruction algorithms we developed. Concluding remarks are offered in Section~\ref{SEC:Concl}.

\section{General Properties of the Inverse Problems}
\label{SEC:Gen}

We review in this section some general properties of the inverse problem of reconstructing $\eta(\bx)$ and/or $\sigma_{a,xf}(\bx)$ in the transport system~\eqref{EQ:ERT QfPAT} from the datum $H$ in~\eqref{EQ:Data QfPAT}. We denote by $L^p(X)$ (resp. $L^p(\Omega)$) the Lebesgue space of real-valued functions whose $p$-th power are Lebesgue integrable on $X$ (resp. $\Omega$), and $\cH_p^1(X)$ the space of $L^p(X)$ functions whose derivative in direction $\bv$ is in $L^p(X)$, i.e. $\cH_p^1(X)=\{f(\bx,\bv): f\in L^p(X)\ \mbox{and}\ \bv\cdot\nabla f\in L^p(X)\}$. We denote by $L^p(\Gamma_-)$ the space of functions that are traces of $\cH_p^1(X)$ functions on $\Gamma_-$ under the norm $\|f\|_{L^p(\Gamma_-)}=(\int_{\partial\Omega}\int_{\bbS_{\bx-}^{d-1}}|\bn(\bx)\cdot\bv||f|^p d\bv d\gamma)^{1/p}$, $d\gamma$ being the surface measure on $\partial\Omega$ and $\bbS_{\bx-}^{d-1}=\{\bv: \bv\in\bbS^{d-1}\ \mbox{s.t.}\ -\bn(\bx)\cdot\bv>0\}$. It is well-known~\cite{Agoshkov-Book98, DaLi-Book93-6} that both $\cH_p^1(X)$ and $L^p(\Gamma_-)$ are well-defined. To avoid confusion with $\cH_p^1(X)$, we use $W_2^k(\Omega)$ to denote the usual Hilbert space of $L^2(\Omega)$ functions whose partial derivatives up to order $k$ are all in $L^2(\Omega)$. Besides the assumptions in ({\bf A-i})-({\bf A-ii}), we assume further that:\vskip 2mm

\noindent ({\bf A-iii}) The domain $\Omega$ is simply-connected with $C^2$ boundary $\partial \Omega$. The known optical coefficients satisfy $0<\sfc_1 \le \sigma_{a,xi}, \sigma_{s,x}, \sigma_{a,m},\sigma_{s,m}, \Xi \le \sfc_2<+\infty$ for some positive constants $\sfc_1$ and $\sfc_2$. The unknown coefficients, $(\eta,\sigma_{a,xf})$ belongs to the class
\begin{equation}\label{EQ:Bound Coeff}
	\cA=\{(\eta,\sigma_{a,xf}): 0<\sfc_3\le \eta \le \sfc_4 <1,\ 0<\sfc_5 \le \sigma_{a,xf} \le \sfc_6<+\infty\}
\end{equation}
for some positive constants $\sfc_3$, $\sfc_4$, $\sfc_5$ and $\sfc_6$. The scattering kernel $\Theta$ is symmetric, bounded and normalized in the sense that
\begin{equation}\label{EQ:Theta Properties}
\begin{array}{c}
\Theta(\bv,\bv')=\Theta(\bv',\bv),\quad 0<\sfc_7\le \Theta(\bv,\bv')\le \sfc_8<+\infty,\ \ \forall \bv,\bv'\in\bbS^{d-1},\\
\dint_{\bbS^{d-1}}\Theta(\bv,\bv') d\bv'=\dint_{\bbS^{d-1}}\Theta(\bv',\bv) d\bv'=1,\ \ \forall \bv\in\bbS^{d-1},
\end{array}
\end{equation}
for some positive constants $\sfc_7$ and $\sfc_8$. The illumination $g_x(\bx,\bv)$ is strictly positive such that $0<\sfc_9\le g_x(\bx,\bv)$ for some $\sfc_9$.\vskip 2mm

With the above settings, it is easy to see, following standard results in~\cite{Agoshkov-Book98,DaLi-Book93-6}, that the system~\eqref{EQ:ERT QfPAT} admits a unique solution in the following sense.
\begin{lemma}\label{LMMA:ERT Stab}
Let $p\in[1,\infty]$ and assume that ({\bf A-iii}) holds. Then for any given function $g_x(\bx,\bv)\in L^p(\Gamma_-)$, there exists a unique solution $(u_x, u_m)\in \cH_p^1(X)\times \cH_p^1(X)$ to the couple transport system~\eqref{EQ:ERT QfPAT}. Moreover, the following bound holds:
\begin{equation}\label{EQ:ERT QfPAT Stab}
	\|u_x\|_{L^p(X)}+\|u_m\|_{L^p(X)}\le \sfc\|g_x\|_{L^p(\Gamma_-)}
\end{equation}
with the constant $\sfc$ depending only on $\Omega$ and the bounds for the coefficients in assumption ({\bf A-iii}).
\end{lemma}
\begin{proof}
When the assumptions are satisfied, it follows directly from standard transport theory in~\cite{Agoshkov-Book98,DaLi-Book93-6} that the first transport equation admits a unique solution $u_x\in \cH_p^1(X)$ such that $\|u_x\|_{L^p(X)} \le \tilde\sfc \|g_x\|_{L^p(\Gamma_-)}$. We then deduce, with the same argument that the second equation admit a unique solution $u_m\in \cH_p^1(X)$ such that $\|u_m\|_{L^p(X)} \le \hat\sfc \|\eta\sigma_{a,xf} K_I(u_x)\|_{L^p(\Omega)}\le \hat{\hat\sfc}\|u_x\|_{L^p(X)}$. The bound in ~\eqref{EQ:ERT QfPAT Stab} then follows from selecting $\sfc=\tilde\sfc(1+\hat{\hat\sfc})$.
\end{proof}

The above lemma ensures that the datum $H$ in~\eqref{EQ:Data QfPAT} is well-defined for any $g_x(\bx,\bv)\in L^p(\Gamma_-)$ ($p\in[1,\infty]$) that satisfies the assumptions in ({\bf A-iii}). Moreover $H\in L^p(\Omega)$ following standard results in~\cite{DaLi-Book93-6}. The next result shows that the datum $H$ depends continuously on the unknown coefficients and is differentiable with respect to the coefficients in appropriate sense.
\begin{proposition}\label{PROP:Data QfPAT Frechet}
	Let $p\in[1,\infty]$ and assume that ({\bf A-iii}) holds. Then for any given function $g_x(\bx,\bv)\in L^p(\Gamma_-)$, the datum $H$ defined in~\eqref{EQ:Data QfPAT}, viewed as the map
	\begin{equation}
	H[\eta,\sigma_{a,xf}]:
	\begin{array}{ccl}
		(\eta,\sigma_{a,xf}) &\mapsto& \Xi\big(\sigma_{a,x}^\eta K_I(u_{x})+\sigma_{a,m} K_I(u_{m})\big)\\
		L^{\infty}(\Omega)\times L^{\infty}(\Omega) &\mapsto & L^p(\Omega)
	\end{array}
	\end{equation}
	is Fr\'echet differentiable at any $(\eta,\sigma_{a,xf})\in L^{\infty}(\Omega)\times L^{\infty}(\Omega)$ in the direction $(\delta\eta,\delta\sigma_{a,xf})\in L^\infty(\Omega)\times L^\infty(\Omega)$ that satisfy $(\eta,\sigma_{a,xf})\in\cA$ and $(\eta+\delta\eta,\sigma_{a,xf}+\delta\sigma_{a,xf})\in\cA$. The derivative is given by
	\begin{equation}\label{EQ:Data QfPAT Linearized}
	H'[\eta,\sigma_{a,xf}](\delta\eta,\delta\sigma_{a,xf}) = \Xi \Big((-\delta\eta\sigma_{a,xf}+(1-\eta)\delta\sigma_{a,xf}) K_I(u_{x}) +\sigma_{a,x}^\eta K_I(v_x)+\sigma_{a,m} K_I(v_{m})\Big)
	\end{equation}
where $(v_x,v_m)\in \cH_p^1(X)\times\cH_p^1(X)$ is the unique solution to
\begin{equation}\label{EQ:ERT QfPAT Linearized}
	\begin{array}{rcll}
  	\bv\cdot\nabla v_x+  \sigma_{t,x} v_x & = & \sigma_{s,x} K_\Theta(v_x)-\delta\sigma_{a,xf} u_x, &\mbox{in}\ \ X\\
  	\bv\cdot\nabla v_m+ \sigma_{t,m} v_m & = & \sigma_{s,m} K_\Theta(v_m) +\eta\sigma_{a,xf}K_I(v_x)+ (\eta\delta\sigma_{a,xf} +\delta \eta\sigma_{a,xf}) K_I(u_x), &\mbox{in}\ \ X\\
    v_x(\bx,\bv) = 0,&& v_m(\bx,\bv) = 0 & \mbox{on}\ \ \Gamma_-
	\end{array}
\end{equation}
where $\sigma_{t,x}=\sigma_{a,x}+\sigma_{s,x}$ and $\sigma_{t,m}=\sigma_{a,m}+\sigma_{s,m}$. 
\end{proposition}
\begin{proof}
Let $\tilde\eta=\eta+\delta\eta$, $\tilde\sigma_{a,xf}=\sigma_{a,xf}+\delta\sigma_{a,xf}$, and define $\Delta(\eta\sigma_{a,xf})=\tilde\eta\tilde\sigma_{a,xf}-\eta\sigma_{a,xf}$. We denote by $(\tilde u_x, \tilde u_m)$ the solution to~\eqref{EQ:ERT QfPAT} with the coefficients $(\tilde\eta, \tilde\sigma_{a,xf})$, and $\tilde H$ the corresponding datum. It is straightforward to verify that $(u_x',u_m')\equiv(\tilde u_x-u_x,\tilde u_m-u_m)$ solves the following system of transport equations
\begin{equation}\label{EQ:ERT Dif}
	\begin{array}{rcll}
  	\bv\cdot\nabla u_x'+ \sigma_{t,x} u_x' & = & \sigma_{s,x} K_\Theta(u_x')-\delta\sigma_{a,xf} \tilde u_x, &\mbox{in}\ \ X\\
  	\bv\cdot\nabla u_m'+ \sigma_{t,m} u_m' & = & \sigma_{s,m} K_\Theta(u_m') +\eta\sigma_{a,xf} K_I(u_x')+F(\bx), & \mbox{in}\ \ X\\
    u_x'(\bx,\bv) = 0,&& u_m'(\bx,\bv) = 0 & \mbox{on}\ \ \Gamma_-
	\end{array}
\end{equation}
with $F(\bx)=\Delta(\eta\sigma_{a,x})K_I(\tilde u_x)$, and $(u_x'',u_m'')\equiv(u_x'-v_x, u_m'-v_m)$ solves the following system
\begin{equation}\label{EQ:ERT Dif2}
	\begin{array}{rcll}
  	\bv\cdot\nabla u_x''+ \sigma_{t,x} u_x'' & = & \sigma_{s,x} K_\Theta(u_x'')-\delta\sigma_{a,xf}u_x', &\mbox{in}\ \ X\\
  	\bv\cdot\nabla u_m''+ \sigma_{t,m} u_m'' & = & \sigma_{s,m} K_\Theta(u_m'')+\eta\sigma_{a,xf} K_I(u_x'')+G(\bx), &\mbox{in}\ \ X\\
    u_x''(\bx,\bv) = 0,&& u_m''(\bx,\bv) = 0 & \mbox{on}\ \ \Gamma_-,
	\end{array}
\end{equation}
with $G(\bx)=\Delta(\eta\sigma_{a,xf}) K_I(u_x')+\delta\eta\delta\sigma_{a,xf}K_I(u_x)$.

With the assumptions on the coefficients and the illumination source $g_x$, we conclude that $(u_x, u_m)\in \cH_p^1(X)\times \cH_p^1(X)$ and $(\tilde u_x, \tilde u_m)\in \cH_p^1(X)\times \cH_p^1(X)$~\cite{Agoshkov-Book98,DaLi-Book93-6}. This implies that $F\in L^p(\Omega)$ and \begin{multline}\label{EQ:F Bound}
	\|F\|_{L^p(\Omega)} = \|(\eta\delta\sigma_{a,x}+\delta\eta\sigma_{a,xf}+\delta\eta\delta\sigma_{a,xf})K_I(\tilde u_x)\|_{L^p(\Omega)} \\ 
\le (\tilde \sfc_1\|\delta\eta\|_{L^\infty(\Omega)}+\tilde\sfc_2\|\delta\sigma_{a,xf}\|_{L^\infty(\Omega)}+\tilde\sfc_3\|\delta\eta\|_{L^\infty(\Omega)}\|\delta\sigma_{a,xf}\|_{L^\infty(\Omega)})\|\tilde u_x\|_{L^p(X)} \\ 
\le (\tilde{\tilde \sfc}_1\|\delta\eta\|_{L^\infty(\Omega)}+\tilde{\tilde \sfc}_2\|\delta\sigma_{a,xf}\|_{L^\infty(\Omega)}+\tilde{\tilde \sfc}_3\|\delta\eta\|_{L^\infty(\Omega)}\|\delta\sigma_{a,xf}\|_{L^\infty(\Omega)})\|g_x\|_{L^p(\Gamma_-)},
\end{multline}
Following the same argument as in Lemma~\ref{LMMA:ERT Stab} we conclude that~\eqref{EQ:ERT Dif} admits a unique solution $(u_x', u_m')\in\cH_p^1(X)\times\cH_p^1(X)$ that satisfies
\begin{equation}\label{EQ:Stab uxprime}
	\|u_x'\|_{L^p(X)} \le \hat\sfc \|\delta\sigma_{a,xf}\tilde u_x\|_{L^p(X)}\le \hat\sfc \|\delta\sigma_{a,xf}\|_{L^\infty(\Omega)}\|\tilde u_x\|_{L^p(X)}\le \hat{\hat\sfc} \|\delta\sigma_{a,xf}\|_{L^\infty(\Omega)}\|g_x\|_{L^p(\Gamma_-)},
\end{equation}
and 
\begin{multline}\label{EQ:Stab umprime}
	\|u_m'\|_{L^p(X)} \le \bar\sfc \big(\|\eta\sigma_{a,xf} K_I(u_x')\|_{L^p(\Omega)}+\|F\|_{L^p(\Omega)}\big) \le \bar{\bar\sfc} \big(\|u_x'\|_{L^p(X)}+\|F\|_{L^p(\Omega)}\big)\\ 
\le (\bar\sfc_1\|\delta\eta\|_{L^\infty(\Omega)}+\bar\sfc_2\|\delta\sigma_{a,xf}\|_{L^\infty(\Omega)}+\bar\sfc_3\|\delta\eta\|_{L^\infty(\Omega)}\|\delta\sigma_{a,xf}\|_{L^\infty(\Omega)})\|g_x\|_{L^p(\Gamma_-)}.
\end{multline}
Therefore we have $G\in L^p(\Omega)$ and the bound
\begin{multline}\label{EQ:G Bound}
	\|G\|_{L^p(\Omega)} \le \|(\eta\delta\sigma_{a,x}+\delta\eta\sigma_{a,xf}+\delta\eta\delta\sigma_{a,xf})K_I(u_x')\|_{L^p(\Omega)}+\|\delta\eta\delta\sigma_{a,xf}K_I(u_x)\|_{L^p(\Omega)} \\
\le (\sfc_1'\|\delta\eta\|_{L^\infty(\Omega)}+ \sfc_2'\|\delta\sigma_{a,xf}\|_{L^\infty(\Omega)}+\sfc_3'\|\delta\eta\|_{L^\infty(\Omega)}\|\delta\sigma_{a,xf}\|_{L^\infty(\Omega)})\|u_x'\|_{L^p(X)} \\ 
+\|\delta\eta\|_{L^\infty(\Omega)}\|\delta\sigma_{a,xf}\|_{L^\infty(\Omega)})\|u_x\|_{L^p(X)}\\ 
\le (\sfc_1''\|\delta\eta\|_{L^\infty(\Omega)}+\sfc_2'\|\delta\sigma_{a,xf}\|_{L^\infty(\Omega)}+\sfc_3'\|\delta\eta\|_{L^\infty(\Omega)}\|\delta\sigma_{a,xf}\|_{L^\infty(\Omega)})\|\delta\sigma_{a,xf}\|_{L^\infty(\Omega)}\|g_x\|_{L^p(\Gamma_-)}.
\end{multline}
We then deduce, in the same manner as above, that ~\eqref{EQ:ERT Dif2} admits a unique solution $(u_x'', u_m'')$ that satisfies
\begin{equation}\label{EQ:Stab ux2prime}
	\|u_x''\|_{L^p(X)} \le \hat\sfc \|\delta\sigma_{a,xf} u_x'\|_{L^p(\Omega)} \le \hat\sfc \|\delta\sigma_{a,xf}\|_{L^\infty(\Omega)}\|u_x'\|_{L^p(X)}\le \hat\sfc\hat{\hat\sfc} \|\delta\sigma_{a,xf}\|_{L^\infty(\Omega)}^2\|g_x\|_{L^p(\Gamma_-)},
\end{equation}
and
\begin{multline}\label{EQ:Stab um2prime}
	\|u_m''\|_{L^p(X)} \le \|\eta\sigma_{a,xf} K_I(u_x'')\|_{L^p(\Omega)}+\|G\|_{L^p(\Omega)}\le \dot\sfc\|u_x''\|_{L^p(X)}+\|G\|_{L^p(\Omega)} \\ 
\le (\dot\sfc_1\|\delta\eta\|_{L^\infty(\Omega)}+\dot\sfc_2\|\delta\sigma_{a,xf}\|_{L^\infty(\Omega)}+\dot\sfc_3\|\delta\eta\|_{L^\infty(\Omega)}\|\delta\sigma_{a,xf}\|_{L^\infty(\Omega)})\|\delta\sigma_{a,xf}\|_{L^\infty(\Omega)}\|g_x\|_{L^p(\Gamma_-)}.
\end{multline}
The estimates~\eqref{EQ:Stab ux2prime} and~\eqref{EQ:Stab um2prime} show that $(u_x,u_m)$ is Fr\'echet differentiable with respect to $\eta$ and $\sigma_{a,xf}$ as a map: $L^\infty(\Omega)\times L^\infty(\Omega) \mapsto L^p(\Omega)\times L^p(\Omega)$ ($p\in[1,\infty]$). Note that $u_x$ is independent of $\eta$, so its derivative with respect to $\eta$ is zero, as can be seen from~\eqref{EQ:Stab ux2prime}.

The differentiability of $H$ with respect to $(\eta,\sigma_{a,xf})$ then follows from the chain rule and the fact that $\sigma_{a,x}^\eta$ is differentiable with respect to $(\eta,\sigma_{a,xf})$. Alternatively, it can also be seen easily from the bounds~\eqref{EQ:Stab uxprime},~\eqref{EQ:Stab ux2prime},~\eqref{EQ:Stab um2prime} and the following algebraic calculation:
\begin{multline}
	H[\tilde\eta, \tilde\sigma_{a,xf}]-H[\eta,\sigma_{a,xf}]-H'[\eta,\sigma_{a,xf}](\delta\eta, \delta\sigma_{a,xf}) \\ 
=\Xi\Big[\sigma_{a,x}^\eta K_I( u_x'')+\sigma_{a,m} K_I(u_m'')  
+(\delta\sigma_{a,xf}-\Delta(\eta\sigma_{a,x})K_I(u_x') -\delta\eta\delta\sigma_{a,xf} K_I(u_x)\Big].
\end{multline}
This completes the proof.
\end{proof}

We will study Born approximation, i.e. linearization, of the inverse problem of QfPAT in Section~\ref{SEC:Two Coeff}. The above result justifies the linearization process. To compute the partial derivative with respect to $\eta$ (resp. $\sigma_{a,xf}$), denoted by $H_\eta'[\eta,\sigma_{a,xf}]$ (resp. $H_\sigma'[\eta,\sigma_{a,xf}]$), we simply set $\delta\sigma_{a,xf}=0$ (resp. $\delta\eta=0$) in~\eqref{EQ:Data QfPAT Linearized} and ~\eqref{EQ:ERT QfPAT Linearized}. It is straightforward to check that
\begin{equation}\label{EQ:Data QfPAT Linearized Eta}
	\frac{H_\eta'[\eta,\sigma_{a,xf}](\delta\eta)}{\Xi \sigma_{a,xf} K_I(u_{x})}=-\delta\eta +\frac{\sigma_{a,m}}{\sigma_{a,xf} K_I(u_{x})} K_I(v_{m}),
\end{equation}
with $v_m\in \cH_p^1(X)$ the unique solution to
\begin{equation}\label{EQ:ERT QfPAT Linearized Eta}
\begin{array}{rcll}
	\bv\cdot\nabla v_m+ \sigma_{t,m} v_m & = & \sigma_{s,m} K_\Theta(v_m) + \delta \eta\sigma_{a,xf} K_I(u_x), &\mbox{in}\ \ X\\
	v_m(\bx,\bv) &=& 0, & \mbox{on}\ \ \Gamma_-,
\end{array}
\end{equation}
and 
\begin{equation}\label{EQ:Data QfPAT Linearized Sigma}
	\frac{H_\sigma'[\eta,\sigma_{a,xf}](\delta\sigma_{a,xf})}{\Xi (1-\eta)K_I(u_{x})} = \delta\sigma_{a,xf} +\frac{\sigma_{a,x}^\eta}{(1-\eta)K_I(u_{x})} K_I(v_x)+\frac{\sigma_{a,m}}{(1-\eta)K_I(u_{x})} K_I(v_{m}),
	\end{equation}
with $(v_x,v_m)\in \cH_p^1(X)\times\cH_p^1(X)$ the unique solution to
\begin{equation}\label{EQ:ERT QfPAT Linearized Sigma}
\begin{array}{rcll}
	\bv\cdot\nabla v_x+ \sigma_{t,x} v_x & = & \sigma_{s,x} K_\Theta(v_x)-\delta\sigma_{a,xf} u_x, &\mbox{in}\ \ X\\
  	\bv\cdot\nabla v_m+ \sigma_{t,m} v_m & = & \sigma_{s,m} K_\Theta(v_m) +\eta\sigma_{a,xf}K_I(v_x)+ \eta\delta\sigma_{a,xf} K_I(u_x), &\mbox{in}\ \ X\\
    v_x(\bx,\bv) = 0,&& v_m(\bx,\bv) = 0 & \mbox{on}\ \ \Gamma_-.
	\end{array}
\end{equation}

The following result is a standard application of the averaging lemma~\cite{DaLi-Book93-6,GoLiPeSe-JFA88,MK-Book98}. It will be useful in Section~\ref{SEC:Two Coeff}.
\begin{lemma}\label{LMMA:Fredholm}
Assume that ({\bf A-iii}) holds. Let $g_x(\bx,\bv)\in L^\infty(\Gamma_-)$ be such that $K_I(u_x)\ge \sfc>0$ for some constant $\sfc$. Then the rescaled linearized data $\frac{H_\sigma'[\eta,\sigma_{a,xf}](\delta\sigma_{a,xf})}{\Xi (1-\eta)K_I(u_{x})}$, viewed as the linear operator
\begin{equation}\label{EQ:Data Lin sigma}
\frac{H_\sigma'[\eta,\sigma_{a,xf}](\delta\sigma_{a,xf})}{\Xi K_I(u_{x})}:
	\begin{array}{ccl}
		\delta\sigma_{a,xf} &\mapsto& (1-\eta)\delta\sigma_{a,xf} +\frac{\sigma_{a,x}^\eta}{K_I(u_{x})} K_I(v_x)+\frac{\sigma_{a,m}}{K_I(u_{x})} K_I(v_{m})\\
		L^2(\Omega) &\mapsto & L^2(\Omega)
	\end{array}
	\end{equation}
	is Fredholm. The same is true for $\frac{H_\eta'[\eta,\sigma_{a,xf}](\delta\eta)}{\Xi K_I(u_{x})}$ if the background coefficient $\sigma_{a,xf}\ge \tilde\sfc >0$ for some $\tilde\sfc$.
\end{lemma}
\begin{proof}
Let us denote by $S_z$ ($z\in\{x,m\}$) the solution operator of the transport equation with coefficients $\sigma_{a,z}$, $\sigma_{s,z}$ and vacuum boundary condition, i.e. $w_z=S_z(f)$ with $w_z$ the solution to:
\begin{equation*}
\bv\cdot\nabla w_z + \sigma_{t,z} w_x-\sigma_{s,z}K_\Theta(w_z)=f,\ \ \mbox{in}\ X, \qquad w_z=0\ \ \mbox{on}\ \Gamma_- .
\end{equation*}
We can then write $K_I(v_x)$ and $K_I(v_m)$ in~\eqref{EQ:Data Lin sigma} respectively as 
\begin{equation}
K_I(v_x)=-\Lambda_x(\delta\sigma_{a,xf}), \quad\mbox{and},\quad K_I(v_m) =-\Lambda_{mx}(\delta\sigma_{a,xf})+\Lambda_m(\eta\delta\sigma_{a,xf})
\end{equation}
where the operators $\Lambda_x$, $\Lambda_m$ and $\Lambda_{mx}$ are defined as
\begin{eqnarray}
\label{EQ:Lambda1}\Lambda_x(\delta\sigma_{a,xf}) \equiv K_I\big(S_x(u_x \delta\sigma_{a,xf})\big),  \qquad \Lambda_m(\delta\sigma_{a,xf})=K_I\big(S_m(K_I(u_x)\delta\sigma_{a,xf} )\big),&&\\
\label{EQ:Lambda2}\Lambda_{mx}(\delta\sigma_{a,xf})\equiv K_I\big(S_m(\eta\sigma_{a,xf}K_I(S_x(u_x \delta\sigma_{a,xf}))\big).&&
\end{eqnarray}
Following the averaging lemma~\cite{DaLi-Book93-6,GoLiPeSe-JFA88,MK-Book98} and the compact embedding of $W_2^{1/2}(\Omega)$ to $L^2(\Omega)$, we conclude $K_I: L^2(X)\to L^2(\Omega)$ is compact. Due to boundedness of $u_x$ (and therefore $K_I(u_x)$), $\eta$ and $\sigma_{a,xf}$, both $S_x$ and $S_m$ are compact as operators from $L^2(\Omega)$ to $L^2(X)$ with the assumptions on the coefficients in ({\bf A-i})~\cite{DaLi-Book93-6,MK-Book98}. Hence, $\Lambda_x$, $\Lambda_m$, and $\Lambda_{mx}$ are all compact operators on $L^2(\Omega)$. Therefore $\frac{H_\sigma'[\eta,\sigma_{a,xf}](\delta\sigma_{a,xf})}{\Xi K_I(u_{x})}$ as an operator can be represented as $(1-\eta)\cI+\cK$ with $\cK$ compact. Therefore it is Fredholm. The same argument works for $\frac{H_\eta'[\eta,\sigma_{a,xf}](\delta\eta)}{\Xi K_I(u_{x})}$.
\end{proof}

\section{Reconstructing of a Single Coefficient}
\label{SEC:One Coeff}

In this section, we consider the reconstruction of one of the two coefficients of interests, assuming the other is known. We start with the reconstruction of the quantum efficiency.

\subsection{The reconstruction of $\eta$}
\label{SUBSEC:Eta}

Assume now that the fluorescence absorption coefficient $\sigma_{a,xf}$ is \emph{known} and we are interested in reconstructing only $\eta$. This is a linear inverse source problem. We can derive the following stability result on the reconstruction.
\begin{theorem}\label{THM:Rec Eta}
	Let $p\in[1,\infty]$ and the source  $g_x\in L^p(\Gamma_-)$ be such that the transport solution $u_x$ to~\eqref{EQ:ERT QfPAT} satisfies $K_I(u_x)\ge \tilde c>0$ for any $(\eta,\sigma_{a,xf})\in\cA$. Let $H$ and $\tilde H$ be two data sets generated with coefficients $(\eta, \sigma_{a,xf})$ and $(\tilde\eta, \sigma_{a,xf})$ respectively. Then $H=\tilde H$ a.e. implies $\eta=\tilde\eta$ a.e.. Moreover, the following stability estimate holds,
	\begin{equation}\label{EQ:Stab eta}
	\sfc\|H-\tilde H\|_{L^p(\Omega)}\le \|(\eta-\tilde \eta) \sigma_{a,xf} K_I(u_x)\|_{L^p(\Omega)}\le \sfC\|H-\tilde H\|_{L^p(\Omega)}
	\end{equation}
	where the constants $\sfc$ and $\sfC$ depend on $\Omega$ and the coefficients $\sigma_{a,xi}$, $\sigma_{a,m}$, $\sigma_{s,x}$, $\sigma_{s,m}$, and $\Xi$. 
\end{theorem}
\begin{proof}
Let $(u_x,u_m)$ and $(\tilde u_x,\tilde u_m)$ be solutions to the coupled transport system~\eqref{EQ:ERT QfPAT} with coefficients $(\eta,\sigma_{a,xf})$ and $(\tilde\eta,\sigma_{a,xf})$ respectively. We notice immediately that $u_x=\tilde u_x$. Define $w_m=u_m-\tilde u_m$. We then verify that
\begin{equation}\label{EQ:Data Difference}
(H-\tilde H)/\Xi=-(\eta-\tilde\eta)\sigma_{a,xf}K_I(u_x)+\sigma_{a,m}K_I(w_m)
\end{equation}
This leads to the bound
\begin{equation}\label{EQ:Stab Eta I}
\|H-\tilde H\|_{L^p(\Omega)}\le \sfc_1 \|(\eta-\tilde\eta)\sigma_{a,xf} K_I(u_x)\|_{L^p(\Omega)}+\sfc_2(\sigma_{a,m})\|K_I(w_m)\|_{L^p(\Omega)} .
\end{equation}
and the bound
\begin{equation}\label{EQ:Stab Eta II}
\|(\eta-\tilde\eta)\sigma_{a,xf} K_I(u_x)\|_{L^p(\Omega)} \le \tilde\sfc_1(\Xi)\|H-\tilde H\|_{L^p(\Omega)}+\tilde\sfc_2(\sigma_{a,m})\|K_I(w_m)\|_{L^p(\Omega)},
\end{equation}
We check also that $w_m$ solves the transport equation
\begin{equation}\label{EQ:ERT IP eta}
\begin{array}{rcll}
	\bv\cdot\nabla w_m+ (\sigma_{a,m}+\sigma_{s,m}) w_m & = & \sigma_{s,m} K_\Theta(w_m) 
		+ (\eta-\tilde\eta)\sigma_{a,xf}K_I(u_x), &\mbox{in}\ \ X\\
	w_m(\bx,\bv) &=& 0,& \mbox{on}\ \ \Gamma_-.
\end{array}
\end{equation}
It then follows from classical results in transport theory~\cite{Agoshkov-Book98,DaLi-Book93-6} that this equation admits a unique solution $w_m\in \cH_p^1(X)$ that satisfies the following stability estimate
\begin{equation}\label{EQ:Stab Eta I A}
\|w_m\|_{L^p(X)} \le \sfc_3(\Omega,\sigma_{a,m},\sigma_{s,m},\Xi)\|(\eta-\tilde\eta)\sigma_{a,xf}K_I(u_x)\|_{L^p(\Omega)}.
\end{equation}
The left bound in~\eqref{EQ:Stab eta} then follows from ~\eqref{EQ:Stab Eta I} and~\eqref{EQ:Stab Eta I A}.

To derive the right bound in~\eqref{EQ:Stab eta}, we replace the last term in the transport equation~\eqref{EQ:ERT IP eta} with $\sigma_{a,m}K_I(w_m)-(H-\tilde H)/\Xi$ to get
\begin{equation}
\begin{array}{rcll}
	\bv\cdot\nabla w_m+ (\sigma_{a,m}+\sigma_{s,m}) w_m & = & \sigma_{a,m}K_I(w_m) + 
		\sigma_{s,m} K_\Theta(w_m) - \frac{H-\tilde H}{\Xi}, &\mbox{in}\ \ X\\
	w_m(\bx,\bv) &=& 0,& \mbox{on}\ \ \Gamma_-.
\end{array}
\end{equation}
We define $\widetilde\Theta(\bx,\bv,\bv')=\frac{\sigma_{a,m}}{\sigma_{a,m}+\sigma_{s,m}}+\frac{\sigma_{s,m}}{\sigma_{a,m}+\sigma_{s,m}}\Theta$. It is straightforward to verify that $\widetilde\Theta$ is symmetric and normalized in the sense of~\eqref{EQ:Theta Properties}. We can then rewrite the above transport equation as
\begin{equation}\label{EQ:ERT Eta Rec F}
\begin{array}{rcll}
	\bv\cdot\nabla w_m+ (\sigma_{a,m}+\sigma_{s,m}) w_m & = & (\sigma_{a,m}+\sigma_{s,m})K_{\widetilde \Theta}(w_m)- \frac{H-\tilde H}{\Xi}, &\mbox{in}\ \ X\\
	w_m(\bx,\bv) &=& 0,& \mbox{on}\ \ \Gamma_-.
\end{array}
\end{equation}
This is a transport equation for a conservative medium. Due to the fact that $\Omega$ is bounded, classical results in transport theory (see for instance~\cite[Theorem 1 on page 337]{DaLi-Book93-6}) then concludes that the equation admits a unique solution $w_m\in \cH_p^1(X)$. Moreover, we have the stability estimate
\begin{equation}\label{EQ:Stab Eta II A}
\|w_m\|_{L^p(X)} \le \sfc_4(\Omega,\sigma_{a,m},\sigma_{s,m},\Xi)\|H-\tilde H\|_{L^p(\Omega)}
\end{equation}
The right bound in~\eqref{EQ:Stab eta} then follows from ~\eqref{EQ:Stab Eta II} and~\eqref{EQ:Stab Eta II A}. 
The uniqueness of the reconstruction then follows from the fact that $H=\tilde H$ implies $w_m=0$ from~\eqref{EQ:ERT Eta Rec F}, which then implies $\eta=\tilde\eta$ from~\eqref{EQ:Data Difference}.
\end{proof}
Note that the bound in~\eqref{EQ:Stab eta} is weighted in the sense that it is on $(\eta-\tilde\eta)K_I(u_x)$ not directly on $(\eta-\tilde\eta)$. This means that if $K_I(u_x)$ is too small, it is very hard to reconstruct accurately $\eta$.

The proof of the above stability result is constructive in the sense that it provides an explicit reconstruction procedure for the recovery of $\eta$. We now summarize the procedure in the following algorithm.
\vskip 2mm
\noindent{\bf Reconstruction Algorithm I.}
\begin{itemize}
	\item[S1.] Given $\sigma_{a,xf}$, solve the first transport equation in~\eqref{EQ:ERT QfPAT} with the boundary condition $g_x$ for $u_x$;
	\item[S2.] Evaluate the function $q(\bx)=\sigma_{a,x} K_I(u_x)-\frac{H}{\Xi}$;
	\item[S3.] Solve the following transport equation for $u_m$:
	\begin{equation}\label{EQ:ERT IP Recon eta}
	\begin{array}{rcll}
  		\bv\cdot\nabla u_m+ (\sigma_{a,m}+\sigma_{s,m}) u_m &=& (\sigma_{a,m}+\sigma_{s,m})K_{\widetilde\Theta}(u_m) + q(\bx), &\mbox{in}\ \ X\\
		u_m(\bx,\bv) &=& 0,& \mbox{on}\ \ \Gamma_-.
	\end{array}
\end{equation}
	\item[S4.] Reconstruct $\eta$ as $-\big(\dfrac{H}{\Xi}-\sigma_{a,x} K_I(u_x)-\sigma_{a,m} K_I(u_m)\big)/(\sigma_{a,xf} K_I(u_x))$.
\end{itemize}
This is a direct reconstruction algorithm in the sense that it does not involve any iteration on the the unknown coefficient. The algorithm is very efficient since it requires solving the transport equation ~\eqref{EQ:ERT IP Recon eta} only once.

\begin{remark}\label{RM:Rec Eta}
Thanks to the fact that the problem of reconstructing $\eta$ given $\sigma_{a,xf}$ is linear, we can easily verify that the same type of uniqueness and stability results in Theorem~\ref{THM:Rec Eta} hold for the linearized problem of reconstructing $\eta$ defined in~\eqref{EQ:ERT QfPAT Linearized Eta} and~\eqref{EQ:Data QfPAT Linearized Eta}. Moreover, the above reconstruction algorithm works in exactly the same manner in the linearized setting.
\end{remark}

\subsection{The reconstruction of $\sigma_{a,xf}$}
\label{SUBSEC:Sigma}

We now assume that we know $\eta$ and aim at reconstructing $\sigma_{a,xf}$. In this case, we can show the following result.
\begin{theorem}\label{THM:Sigma}
Let $g_x\in L^p(\Gamma_-)$ ($p\in[1,\infty]$) be such that the solution $u_x$ to the transport system~\eqref{EQ:ERT QfPAT} satisfies $u_x=K_I(u_x) \ge \tilde \sfc> 0$ for any coefficient pair $(\eta, \sigma_{a,xf})\in \cA$. Let $H$ and $\tilde H$ be data sets generated with coefficient pairs $(\eta, \sigma_{a,xf})$ and $(\eta, \tilde\sigma_{a,xf})$ respectively. Then $H=\tilde H$ a.e. implies $\sigma_{a,xf}=\tilde\sigma_{a,xf}$ a.e.. Moreover, the following bound holds,
\begin{equation}\label{EQ:Stab sigma}
	 \sfc \|H-\tilde H\|_{L^p(\Omega)}\le \|(\sigma_{a,xf}-\tilde \sigma_{a,xf}) K_I(u_x)\|_{L^p(\Omega)} \le \sfC \|H-\tilde H\|_{L^p(\Omega)},
\end{equation}
with $\sfc$ and $\sfC$ depending on $\Omega$, $\sigma_{a,xi}$, $\sigma_{a,m}$, $\sigma_{s,x}$, $\sigma_{s,m}$, $\eta$ and $\Xi$.
\end{theorem}
\begin{proof}
Let $(u_x,u_m)$ and $(\tilde u_x,\tilde u_m)$ be solutions to the coupled transport system~\eqref{EQ:ERT QfPAT} with coefficients $(\eta,\sigma_{a,xf})$ and $(\eta,\tilde \sigma_{a,xf})$ respectively. Define $w_x=u_x-\tilde u_x$ and $w_m=u_m-\tilde u_m$. Then we have
\begin{equation}\label{EQ:Data QfPAT Diff}
\frac{H-\tilde H}{\Xi}=\tilde\sigma_{a,x}^\eta K_I(w_x)+\sigma_{a,m}K_I(w_m)+(1-\eta)(\sigma_{a,xf}-\tilde \sigma_{a,xf}) K_I(u_x).
\end{equation}
This leads to the bound
\begin{equation}\label{EQ:Stab I Sig}
\|H-\tilde H\|_{L^p(\Omega)}\le \sfc_1'\| K_I(w_x)\|_{L^p(\Omega)}+\sfc_2'\|K_I(w_m)\|_{L^p(\Omega)}+ \sfc_3'\|(\sigma_{a,xf}-\tilde\sigma_{a,xf})K_I(u_x)\|_{L^p(\Omega)},
\end{equation}
and the bound
\begin{equation}\label{EQ:Stab I Sig A}
\|(\sigma_{a,xf}-\tilde \sigma_{a,xf})K_I(u_x)\|_{L^p(\Omega)}\le \sfc_1''\|H-\tilde H\|_{L^p(\Omega)}+\sfc_2''\| K_I(w_x)\|_{L^p(\Omega)}+\sfc_3''\|K_I(w_m)\|_{L^p(\Omega)}.
\end{equation}
We now verify that $(w_x, w_m)$ solves the following transport system:
\begin{equation}\label{EQ:ERT Stab Sigma}
	\begin{array}{rcll}
  	\bv\cdot\nabla w_x+ \tilde\sigma_{t,x} w_x & = & \sigma_{s,x} K_\Theta(w_x) -(\sigma_{a,xf}-\tilde \sigma_{a,xf}) u_x, &\mbox{in}\ \ X\\
  	\bv\cdot\nabla w_m+ \sigma_{t,m} w_m & = & \sigma_{s,m} K_\Theta(w_m)+\eta\tilde\sigma_{a,xf}K_I(w_x)+\eta(\sigma_{a,xf}-\tilde\sigma_{a,xf})K_I(u_x), &\mbox{in}\ \ X\\
    w_x(\bx,\bv) = 0,&& w_m(\bx,\bv) = 0, & \mbox{on}\ \ \Gamma_-
	\end{array}
\end{equation}
where $\sigma_{t,x}=\sigma_{a,xi}+\tilde\sigma_{a,xf}+\sigma_{s,m}$. We then deduce, following similar procedure as in the proof of Proposition~\ref{PROP:Data QfPAT Frechet}, that 
\begin{equation}\label{EQ:Stab II Sig}
	\|w_x\|_{L^p(X)}+\|w_m\|_{L^p(X)} \le \sfc_4'\|(\sigma_{a,xf}-\tilde\sigma_{a,xf})K_I(u_x)\|_{L^p(\Omega)}.
\end{equation}
The left bound in~\eqref{EQ:Stab sigma} then follows from~\eqref{EQ:Stab I Sig} and~\eqref{EQ:Stab II Sig}.

To derive the right bound in~\eqref{EQ:Stab sigma}, we use~\eqref{EQ:Data QfPAT Diff} to eliminate the quantity $\sigma_{a,xf}-\tilde\sigma_{a,xf}$ in the transport system~\eqref{EQ:ERT Stab Sigma} to obtain:
\begin{equation}\label{EQ:ERT IP Sigma}
	\begin{array}{rcll}
  	\bv\cdot\nabla w_x+ \tilde\sigma_{t,x} w_x & = & \sigma_{s,x} K_\Theta(w_x)+\sigma'_{s,x} K_I(w_x)+\sigma'_{s,xm}K_{I}(w_m)-\frac{(H-\tilde H)u_x}{\Xi(1-\eta)K_I(u_x)}, &\mbox{in}\ \ X\\
  	\bv\cdot\nabla w_m+ \sigma_{t,m} w_m & = & \sigma_{s,m} K_\Theta(w_m)-\sigma'_{s,m}K_{I}(w_m)-\sigma'_{s,mx} K_I(w_x)+\frac{(H-\tilde H)\eta}{\Xi(1-\eta)}, &\mbox{in}\ \ X\\
    w_x(\bx,\bv) = 0,&& w_m(\bx,\bv) = 0, & \mbox{on}\ \ \Gamma_-
	\end{array}
\end{equation}
where $\sigma'_{s,x}=\frac{\tilde\sigma_{a,x}^\eta u_x}{(1-\eta)K_I(u_x)}$, $\sigma_{s,xm}'=\frac{\sigma_{a,m} u_x}{(1-\eta)K_I(u_x)}$, $\sigma'_{s,m}=\frac{\eta\sigma_{a,m}}{1-\eta}$, and $\sigma'_{s,mx}=\frac{\eta\sigma_{a,xi}}{1-\eta}$. To write the system in standard form, we perform the change of variable $w_x\to -w_x$. We then have
\begin{equation}\label{EQ:ERT IP Sigma II}
	\begin{array}{rcll}
  	\bv\cdot\nabla w_x+ \tilde\sigma_{t,x} w_x +\sigma'_{s,xm}K_{I}(w_m) & = & \sigma_{s,x} K_\Theta(w_x)+\sigma'_{s,x} K_I(w_x)+\frac{(H-\tilde H)u_x}{\Xi(1-\eta)K_I(u_x)}, &\mbox{in}\ \ X\\
  	\bv\cdot\nabla w_m+ \sigma_{t,m} w_m +\sigma'_{s,m}K_{I}(w_m)& = & \sigma_{s,m} K_\Theta(w_m)+\sigma'_{s,mx} K_I(w_x)+\frac{(H-\tilde H)\eta}{\Xi(1-\eta)}, &\mbox{in}\ \ X\\
    w_x(\bx,\bv) = 0,&& w_m(\bx,\bv) = 0, & \mbox{on}\ \ \Gamma_-
	\end{array}
\end{equation}
With the assumption on $g_x$, the coefficients $\sigma'_{s,x}$, $\sigma_{s,xm}'$, $\sigma'_{s,m}$, and $\sigma'_{s,mx}$ are all positive. We check also, after using the assumption $u_x=K_I(u_x)$, that $\Delta_1\equiv \tilde\sigma_{t,x}+\sigma_{s,xm}'-\sigma_{s,x}-\sigma_{s,x}'=\tilde\sigma_{a,x}+(\sigma_{a,m}-\tilde\sigma_{a,x}^\eta)/[(1-\eta)]$ and $\Delta_2\equiv \sigma_{t,m}+\sigma_{s,m}'-\sigma_{s,m}-\sigma_{s,mx}'=(\sigma_{a,m}-\eta\sigma_{a,xi})/(1-\eta)$. The conditions in Theorem~\ref{THM:Sigma} ensure that $\Delta_1, \Delta_2\ge \sfc'>0$ for some $\sfc'$. We can therefore combine the techniques in~\cite{GrSa-JMP76,Tervo-JMAA07,TeKo-arXiv14,WiMe-JMP86}, see detailed analysis in~\cite{Ren-Prep15}, to show that system~\eqref{EQ:ERT IP Sigma II} admits a unique solution $(w_x, w_m)$ that satisfies
\begin{equation}\label{EQ:Stab II Sig A}
	\|w_x\|_{L^p(X)}+\|w_m\|_{L^p(X)} \le \sfc_4''\|H-\tilde H\|_{L^p(\Omega)}.
\end{equation}
We can now combine ~\eqref{EQ:Stab I Sig A} and~\eqref{EQ:Stab II Sig A} to obtain the right bound in~\eqref{EQ:Stab sigma}. The uniqueness result follows from the fact that~\eqref{EQ:ERT IP Sigma II} admits only the trivial solution $(w_x, w_m)=(0, 0)$ when $H=\tilde H$.
\end{proof}

\paragraph{Linearized Case.} Unlike in the case of reconstructing $\eta$, the above proof is not constructive since the unknown coefficient $\sigma_{a,xf}$ shows up in the transport system~\eqref{EQ:ERT IP Sigma II}. Therefore, the proof does not provide directly a reconstruction algorithm. For numerical reconstructions in this nonlinear setting, we use the optimization-based algorithm in Section~\ref{SUBSEC:Nonl Rec}. If we consider the same problem in linearized setting, we can indeed derive an explicit reconstruction procedure. To do that, we replace the $\delta\sigma_{a,xf}$ in~\eqref{EQ:ERT QfPAT Linearized Sigma} with its expression given in the linearized datum~\eqref{EQ:Data QfPAT Linearized Sigma} to get the following system:
\begin{equation}\label{EQ:ERT IP Sigma Lin}
\begin{array}{rcll}
	\bv\cdot\nabla v_x+ \sigma_{t,x} v_x +\sigma_{s,xm}'K_I(v_m) & = & \sigma_{s,x} K_\Theta(v_x)+\sigma_{s,x}'K_I(v_x)-\frac{u_x H_\sigma'}{(1-\eta)\Xi K_I(u_x)}, &\mbox{in}\ \ X\\
  	\bv\cdot\nabla v_m+ \sigma_{t,m} v_m +\sigma_{s,m}' K_I(v_m) & = & \sigma_{s,m} K_\Theta(v_m) +\sigma'_{s,mx} K_I(v_x) + \frac{\eta H_\sigma'}{(1-\eta)\Xi}, &\mbox{in}\ \ X\\
    v_x(\bx,\bv) = 0,&& v_m(\bx,\bv) = 0, & \mbox{on}\ \ \Gamma_-
	\end{array}
\end{equation}
where we have performed the change of variable $v_x\to -v_x$, and the coefficient $\sigma'_{s,x}=\frac{\sigma_{a,x}^\eta u_x}{(1-\eta)K_I(u_x)}$, while the coefficients $\sigma_{s,xm}'$, $\sigma'_{s,m}$, and $\sigma'_{s,mx}$ are defined as in~\eqref{EQ:ERT IP Sigma}. This system does not contain the unknown coefficient $\delta\sigma_{a,xf}$. It can be solved for $(v_x, v_m)$. We can then reconstruct $\delta\sigma_{a,xf}$ following~\eqref{EQ:Data QfPAT Linearized Sigma}. The reconstruction procedure can be summarized into the following reconstruction algorithm.\vskip 2mm
\noindent{\bf Reconstruction Algorithm II.}
\begin{itemize}
	\item[S1.] Given the background coefficient $\sigma_{a,xf}$, solve the first transport equation in~\eqref{EQ:ERT QfPAT} with the boundary condition $g_x$ for $u_x$ (and therefore $K_I(u_x)$);
	\item[S2.] Evaluate the coefficients $\sigma'_{s,x}$, $\sigma_{s,xm}'$, $\sigma'_{s,m}$ and $\sigma'_{s,mx}$;
	\item[S3.] Solve the transport system~\eqref{EQ:ERT IP Sigma Lin} for $(v_x, v_m)$ and perform the transform $(-v_x, v_m)\to(v_x, v_m)$;
	\item[S4.] Reconstruct $\delta\sigma_{a,xf}$ as $\big[\dfrac{H_\sigma'}{\Xi}-\sigma_{a,x}^\eta K_I(v_x)-\sigma_{a,m} K_I(v_m)\big]/\big[(1-\eta)K_I(u_x)\big]$.
\end{itemize}
Following the control theory for transport equations developed in~\cite{Acosta-IP13,BaAg-ESAIM00,KlYa-SIAM07}, we can show, under reasonable assumptions, the existence of sources $g_x$ such that $u_x=K_I(u_x)$ holds for each pair $(\eta,\sigma_{a,xf})\in\cA$. Such sources, however, might be complicated, for instance we might need to solve a control problem, to construct in practical applications. The usefulness of Reconstruction Algorithm II is therefore limited by this fact. Note that in applications where the medium is scattering-free, see for instance discussions in~\cite{ElScSc-IP12,MaRe-CMS14}, this algorithm is indeed very useful since there are many ways to construct illuminations sources to have $u_x=K_I(u_x)$.
 
\section{Simultaneous Reconstruction of Two Coefficients}
\label{SEC:Two Coeff}

We now consider the problem of simultaneous reconstruction of the quantum efficiency and the fluorescence absorption coefficient. We start with the linearized case.

\subsection{Linearization around $(\eta,\sigma_{a,xf})=(0,0)$}
\label{SUBSEC:Lin 0 0}

We first consider the special case where both coefficients are small. In this case the product of the coefficient is small so that generation of fluorescence is very small and can be neglected. Therefore, the system involves only the light at the excitation wavelength. The QfPAT problem reduces to the usual quantitative PAT problem. To be precise, we linearize the problem around the background $(\eta,\sigma_{a,xf})=(0,0)$. Then the second transport equation in~\eqref{EQ:ERT QfPAT Linearized} has the solution $v_m=0$. Therefore, the datum~\eqref{EQ:Data QfPAT Linearized} simplifies to
\begin{equation}\label{EQ:Data QfPAT Pert 0 0}
	\frac{1}{\Xi} H'[0,0](\delta\eta,\delta\sigma_{a,xf})=\delta\sigma_{a,xf} K_I(u_{x})+\sigma_{a,xi}K_I(v_x),
\end{equation}
and the first transport equation in system~\eqref{EQ:ERT QfPAT Linearized} simplifies to
\begin{equation}\label{EQ:ERT QfPAT Linearized A 0 0}
	\begin{array}{rcll}
  	\bv\cdot\nabla v_x+ (\sigma_{a,xi}+\sigma_{s,x}) v_x & = & 
		\sigma_{s,x} K_\Theta(v_x)-\delta\sigma_{a,xf} u_x, &\mbox{in}\ \ X \\
    	v_x(\bx,\bv) & =& 0, & \mbox{on}\ \ \Gamma_-.
	\end{array}
\end{equation}
We observe that $\delta\eta$ does not appear in the datum~\eqref{EQ:Data QfPAT Pert 0 0} or the equation~\eqref{EQ:ERT QfPAT Linearized A 0 0}. Therefore, it can \emph{not} be reconstructed in this setting. We can show the following result.
\begin{proposition}\label{PROP:Unique Sigma Linear}
Let $u_x$ be the solution to the first transport equation in~\eqref{EQ:ERT QfPAT} with $\sigma_{a,xf}=0$. Let $g_x\in L^p(\Gamma_-)$ ($p\in[1,\infty]$) be such that $u_x=K_I(u_x) \ge \tilde \sfc> 0$. Denote by $H'[0, 0]$ and $\widetilde{H}'[0, 0]$ the perturbed data sets in the form of ~\eqref{EQ:Data QfPAT Pert 0 0}, generated with perturbed coefficients ($\delta\eta$, $\delta\sigma_{a,xf}$) and ($\widetilde{\delta\eta}$, $\widetilde{\delta\sigma}_{a,xf}$) respectively. Then $H'[0, 0]=\widetilde{H}'[0, 0]$ a.e. implies $\delta\sigma_{a,xf}=\widetilde{\delta\sigma}_{a,xf}$ a.e.. In addition, we have,
\begin{equation}\label{EQ:Stab Sigma Pert 0 0}
	\sfc\|H'[0, 0]-\widetilde{H}'[0, 0]\|_{L^p(\Omega)}\le \|(\delta\sigma_{a,xf}-\widetilde{\delta\sigma}_{x,f}) K_I(u_x)\|_{L^p(\Omega)}\le \sfC\|H'[0, 0]-\widetilde{H}'[0, 0]\|_{L^p(\Omega)},
\end{equation}
with $\sfc$ and $\sfC$ constants that depend on $\Omega$, $\Xi$, $\sigma_{a,xi}$ and $\sigma_{s,x}$.
\end{proposition}
\begin{proof}
The datum~\eqref{EQ:Data QfPAT Pert 0 0} implies directly that
\begin{equation}\label{EQ:Stab Sigma I}
	\|H'[0, 0]-\widetilde{H}'[0, 0]\|_{L^p(\Omega)} \le \sfc_1\|(\delta\sigma_{a,xf}-\widetilde{\delta\sigma}_{a,xf}) K_I(u_x)\|_{L^p(\Omega)} + \sfc_2 \|v_x-\tilde v_x\|_{L^p(X)},
\end{equation}
and
\begin{equation}\label{EQ:Stab Sigma I A}
	\|(\delta\sigma_{a,xf}-\widetilde{\delta\sigma}_{a,xf}) K_I(u_x)\|_{L^p(\Omega)} \le \sfc_1'\|H'[0, 0]-\widetilde{H}'[0, 0]\|_{L^p(\Omega)} + \sfc_2'\|v_x-\tilde v_x\|_{L^p(X)},
\end{equation}
with the constants depend on $\Omega$, $\sigma_{a,xi}$ and $\Xi$.

With the assumptions in the theorem, we deduce from the transport equation~\eqref{EQ:ERT QfPAT Linearized A 0 0} that
\begin{equation}\label{EQ:Stab Sigma II}
	\|v_x-\tilde v_x\|_{L^p(\Omega)}\le \sfc_3 \|(\delta\sigma_{a,xf}-\widetilde{\delta\sigma}_{a,xf}) K_I(u_x)\|_{L^p(\Omega)}.
\end{equation}
The left bound in~\eqref{EQ:Stab Sigma Pert 0 0} then follows from~\eqref{EQ:Stab Sigma I} and~\eqref{EQ:Stab Sigma II}. To get the right bound in~\eqref{EQ:Stab Sigma Pert 0 0}, we use the datum~\eqref{EQ:Data QfPAT Pert 0 0}, and the assumption that $u_x=K_I(u_x)$, to rewrite~\eqref{EQ:ERT QfPAT Linearized A 0 0} as
\begin{equation}\label{EQ:Diff QfPAT Pert 0 0 Rec}
	\begin{array}{rcll}
  	\bv\cdot\nabla v_x+ (\sigma_{a,xi}+\sigma_{s,x}) v_x & = & 
		\sigma_{s,x} K_\Theta(v_x)+\sigma_{a,xi}K_I(v_x)-\frac{H'[0,0]}{\Xi}, &\mbox{in}\ \ X \\
    	v_x(\bx,\bv) & =& 0, & \mbox{on}\ \ \Gamma_-.
	\end{array}
\end{equation}
This is a conservative transport equation that admits a unique solution with the stability result:
\begin{equation}\label{EQ:Stab Sigma II A}
	\|v_x-\tilde v_x\|_{L^p(\Omega)}\le \sfc_3' \|H'[0, 0]-\widetilde{H}'[0, 0]\|_{L^p(\Omega)},
\end{equation}
where $\sfc_3'$ depends on $\Omega$, $\sigma_{a,xi}$, $\sigma_{s,x}$ and $\Xi$. The right bound in~\eqref{EQ:Stab Sigma Pert 0 0} then follows from~\eqref{EQ:Stab Sigma I A} and ~\eqref{EQ:Stab Sigma II A}.
\end{proof}
The above proof is again constructive when a $g_x$ that satisfies the assumption in the theorem is available to us, in the sense that we only need to solve~\eqref{EQ:Diff QfPAT Pert 0 0 Rec} for $v_x$ and then compute $\delta\sigma_{a,xf}= (H'[0, 0]/\Xi-\sigma_{a,xi} K_I(v_x))/K_I(u_x)$.

\subsection{Linearization around a general background}
\label{SUBSEC:Lin Eta Sigma}

We now consider the linearization around a general background $(\eta\not\equiv 0, \sigma_{a,xf}\not\equiv 0)$. We study the case where we have $J\ge 2$ data sets, $1\le j\le J$:
\begin{multline}\label{EQ:Data QfPAT Pert}
\frac{H_j'[\eta,\sigma_{a,xf}](\delta\eta,\delta\sigma_{a,xf})}{\Xi K_I(u_{x}^j)} = (-\delta\eta\sigma_{a,xf}+(1-\eta)\delta\sigma_{a,xf})  +\frac{\sigma_{a,x}^\eta}{K_I(u_{x}^j)} K_I(v_x^j)+\frac{\sigma_{a,m}}{K_I(u_{x}^j)} K_I(v_{m}^j)
\end{multline}
where $u_x^j$ is the solution to the first transport equation in~\eqref{EQ:ERT QfPAT} with background coefficient $\sigma_{a,xf}$ and illumination source $g_x^j$, while $(v_x^j, v_m^j)$ is the solution to the coupled system~\eqref{EQ:ERT QfPAT Linearized}.

To study the linear inverse problem defined in~\eqref{EQ:Data QfPAT Pert}, we introduce two new variables $\zeta=\delta\eta\sigma_{a,xf}+\eta \delta\sigma_{a,xf}$ and $\xi=\delta \sigma_{a,xf}$. It is straightforward to verify that $(\zeta, \xi)$ uniquely determines $(\delta\eta, \delta\sigma_{a,xf})$ when $\eta\not\equiv 0$ and $\sigma_{a,xf}\not\equiv 0$. We can then collect the $J$ data sets to have the following linear system for the unknown coefficient pair $(\zeta, \xi)$:
\begin{equation}\label{EQ:Lin MATRIX}
\bPi\left (
	\begin{array}{c}
	\zeta\\
	\xi
	\end{array}\right)
	=\bz,\quad \mbox{with},\quad
	\bPi=\left (
	\begin{array}{cc}
	-\cI + \Pi_\zeta^1 & \cI-\Pi_\xi^1\\
	\vdots & \vdots\\
	-\cI + \Pi_\zeta^J & \cI-\Pi_\xi^J
	\end{array}\right)
	\quad\mbox{and}\quad
	\bz
	=
	\left (
	\begin{array}{c}
	\frac{H_1'[\eta,\sigma_{a,xf}]}{\Xi K_I(u_x^1)}\\
	\vdots\\
	\frac{H_J'[\eta,\sigma_{a,xf}]}{\Xi K_I(u_x^J)}
	\end{array}\right)
\end{equation}
with $\Pi_\zeta^j=\frac{\sigma_{a,m}}{K_I(u_x^j)} \Lambda_m^j$ and $\Pi_\xi^j=\frac{\sigma_{a,x}^\eta}{K_I(u_{x}^j)} \Lambda_x^j + \frac{\sigma_{a,m}}{K_I(u_{x}^j)} \Lambda_{mx}^j$. Here $\Lambda_x^j$, $\Lambda_{mx}^j$ and $\Lambda_m^j$ are defined as in~\eqref{EQ:Lambda1} and ~\eqref{EQ:Lambda2} with $u_x$ replaced by $u_x^j$. From Lemma~\ref{LMMA:Fredholm} we know that $\Pi_\zeta^j$ and $\Pi_\xi^j$ ($1\le j\le J$) are compact operators on $L^2(\Omega)$.

From the discussion in the previous sections, we know that $\cI-\Pi_\zeta^j$ and $\cI-\Pi_\xi^j$ are all invertible for well-selected illumination sources $g_x^j$, $1\le j\le J$. However, that does not guarantee the invertibility of the linear system~\eqref{EQ:Lin MATRIX}. For the case of $J=2$, the invertibility of the system~\eqref{EQ:Lin MATRIX} is equivalent to the invertibility of $(\cI-\Pi_\zeta^2)^{-1}(\cI-\Pi_\xi^2)-(\cI-\Pi_\zeta^1)^{-1}(\cI-\Pi_\xi^1)$. Therefore, we need to choose illumination sources $g_x^1$ and $g_x^2$ such that $(\cI-\Pi_\zeta^2)^{-1}(\cI-\Pi_\xi^2)-(\cI-\Pi_\zeta^1)^{-1}(\cI-\Pi_\xi^1)$ is invertible; see next section for some discussions on the regularized version of this problem.

\subsection{A partially linearized model}
\label{SUBSEC:Lin Eta 0}

We now briefly discuss a very popular simplification of the mathematical model in the fluorescence optical tomography literature. This simplification assumes that the fluorescence absorption coefficient $\sigma_{a,xf}$ is small compared to the background tissue absorption coefficient $\sigma_{a,xi}$. Therefore, it can be dropped from the first equation in the model~\eqref{EQ:ERT QfPAT}; see for instance~\cite{PaWaScMa-OL08}. In other words, the model, for source $g_x^j$ ($1\le j\le J$), now reads,
\begin{equation}\label{EQ:ERT QfPAT Partial L}
	\begin{array}{rcll}
  	\bv\cdot\nabla u_x^j+ (\sigma_{a,xi}+ \sigma_{s,x}) u_x^j & = & \sigma_{s,x} K_\Theta(u_x^j), &\mbox{in}\ \ X\\
  	\bv\cdot\nabla u_m^j+ (\sigma_{a,m}+\sigma_{s,m}) u_m^j & = & \sigma_{s,m} K_\Theta(u_m^j) +\eta \sigma_{a,xf} K_I(u_x^j), &\mbox{in}\ \ X\\
    u_x^j(\bx,\bv) = g_x^j,&& u_m^j(\bx,\bv) = 0 & \mbox{on}\ \ \Gamma_-.
	\end{array}
\end{equation}
The data, for source $g_x^j$ ($1\le j\le J$), now simplify to,
\begin{equation}\label{EQ:Data QfPAT Pert Eta 0}
	\widetilde H_j\equiv \frac{H_j}{\Xi K_I(u_x^j)}-\sigma_{a,xi}=(1-\eta)\sigma_{a,xf} +\frac{\sigma_{a,m}}{K_I(u_x^j)} K_I(u_{m}^j).
\end{equation}

The inverse problem of reconstructing $\eta$ and $\sigma_{a,xf}$ from datum~\eqref{EQ:Data QfPAT Pert Eta 0} is a nonlinear problem despite the fact that a partial linearization has been performed on the transport model. However, if we define $\zeta=(1-\eta) \sigma_{a,xf}$ and $\xi=\sigma_{a,xf}$, then the inverse problem is bilinear with respect to $(\zeta, \xi)$. Precisely, we can write the datum as,
\begin{equation}\label{EQ:Data QfPAT Pert Eta 0 j}
	\widetilde H_j= \zeta-\Pi_\zeta^j(\zeta) + \Pi_\zeta^j(\xi), \qquad 1\le j\le J
\end{equation}
with $\Pi_\zeta^j=\frac{\sigma_{a,m}}{K_I(u_x^j)}\Lambda_m^j$ defined the same way as before and being compact on $L^2(\Omega)$. This can again be written into the form of linear system~\eqref{EQ:Lin MATRIX} with the coefficient matrix and source vector respectively
\begin{equation}\label{EQ:Lin MATRIX II}
	\bPi=\left (
	\begin{array}{cc}
		\cI-\Pi_\zeta^1 & \Pi_\zeta^1 \\
		\vdots & \vdots\\
		\cI-\Pi_\zeta^J & \Pi_\zeta^J 
	\end{array}\right), \qquad \mbox{and},\qquad 
	\bz=\left (
	\begin{array}{c}
		\widetilde H_1 \\
		\vdots\\
		\widetilde H_J 
	\end{array}\right).
\end{equation}

\paragraph{Regularized Inversion with $J=2$.} In the case that two data sets are available, we can solve the inverse problems in this section and Section.~\ref{SUBSEC:Lin Eta Sigma} in regularized form. To do that, we observe that if we define 
\begin{equation}
\bPi_\alpha =\bPi+
	\left (
	\begin{array}{cc}
		0 & 0 \\
		0 & \alpha \cI 
	\end{array}\right), \qquad \alpha>0
\end{equation}
then $\bPi_\alpha$ is a Fredholm operator on $L^2(\Omega) \times L^2(\Omega)$ for the $\bPi$ defined in both~\eqref{EQ:Lin MATRIX} and ~\eqref{EQ:Lin MATRIX II}. To be precise, $\bPi_\alpha$ are respectively,
\begin{multline}
\bPi_\alpha=
	\left (
	\begin{array}{cc}
	-\cI + \Pi_\zeta^1 & \cI-\Pi_\xi^1\\
	-\cI + \Pi_\zeta^2 & \alpha\cI+\cI-\Pi_\xi^2
	\end{array}\right)
	\sim
	\left (
	\begin{array}{cc}
	-\cI + \Pi_\zeta^1 & \cI-\Pi_\xi^1\\
	\Pi_\zeta^2-\Pi_\zeta^1 & \alpha\cI+\Pi_\xi^1-\Pi_\xi^2
	\end{array}\right) \\
	=
	\left (
	\begin{array}{cc}
	-\cI + \Pi_\zeta^1 & \cI\\
	0 & \alpha\cI+\Pi_\xi^1
	\end{array}\right)
	+
	\left (
	\begin{array}{cc}
	 0 & -\Pi_\xi^1\\
	\Pi_\zeta^2-\Pi_\zeta^1 & -\Pi_\xi^2
	\end{array}\right),
\end{multline}
and
\begin{equation}
\bPi_\alpha=
	\left (
	\begin{array}{cc}
	\cI - \Pi_\zeta^1 & \Pi_\zeta^1\\
	\cI - \Pi_\zeta^2 & \alpha\cI+\Pi_\zeta^2
	\end{array}\right)
	=
	\left (
	\begin{array}{cc}
	\cI-\Pi_\zeta^1 & 0\\
	\cI & \alpha\cI
	\end{array}\right)
	+
	\left (
	\begin{array}{cc}
	0 & \Pi_\zeta^1\\
	-\Pi_\zeta^2 & \Pi_\zeta^2
	\end{array}\right)
\end{equation}
where $\sim$ is used to denote the elementary operation of subtracting the first row from the second row. For any fixed $\alpha>0$, let us denote by ${\cN}(\bPi_\alpha)$ the null space of matrix operator $\bPi_\alpha$, then the following result follows immediately from classical stability theory of Fredholm operators~\cite{Kato-Book13}.
\begin{proposition}\label{THM:Unique Two}
	Let $\bz$ and $\widetilde{\bz}$ be two perturbed data sets defined as in ~\eqref{EQ:Lin MATRIX} or~\eqref{EQ:Lin MATRIX II}. Let $(\zeta, \xi)^t$ and $(\widetilde \zeta, \widetilde \xi)^t$ be the solution to $\bPi_\alpha \left (\begin{array}{c} \zeta\\ \xi \end{array}\right) =\bz$ and $\bPi_\alpha \left (\begin{array}{c} \widetilde\zeta\\ \widetilde \xi \end{array}\right) =\widetilde\bz$ respectively for some $\alpha>0$. Then we have
\begin{equation}
	\tilde\sfc\|\bz-\widetilde{\bz}\|_{(L^2(\Omega))^2}\le \|(\zeta, \xi)-(\widetilde\zeta, \widetilde\xi) \|_{(L^2(\Omega))^2/{\cN}(\bPi_\alpha)}\le \tilde\sfC\|\bz-\widetilde\bz\|_{(L^2(\Omega))^2}.
\end{equation}
for some constants $\tilde\sfc$ and $\tilde\sfC$.
\end{proposition}

In the numerical computation, to solve ~\eqref{EQ:Lin MATRIX} or ~\eqref{EQ:Lin MATRIX II} directly, we have to construct the operator $\bPi$ explicitly. This is hard to do in practice since it essentially requires the analytical form of the Green's function for the transport equation at the emission wavelength. We do not have access to this Green's function. Instead, solve the linear problem with a classical method of Landweber iteration~\cite{Kirsch-Book11} that we summarize in the following algorithm.\vskip 2mm
\noindent{\bf Reconstruction Algorithm III.}
\begin{itemize}
	\item[S1.] Take initial guess $(\zeta_0, \xi_0)$;
	\item[S2.] Iteratively update the unknown through the iteration:
	\begin{equation}
		\left(
		\begin{array}{c}
		\zeta_{k+1}\\
		\xi_{k+1}
		\end{array}
		\right)
		=(\bI-\tau\bPi^*\bPi) 	
		\left(
		\begin{array}{c}
		\zeta_{k}\\
		\xi_{k}
		\end{array}
		\right)
		+\tau \bPi^*\bz,\quad k\ge 0.
	\end{equation}
	Stop the iteration when desired convergence criteria are satisfied.
\end{itemize}
Here $\tau$ is a positive algorithmic parameter that we select by trial and error. The adjoint operator $\bPi^*$ is formed by transposing $\bPi$ and replacing $\Pi_\zeta^j$ and $\Pi_\xi^j$ with $\Pi_\zeta^{j*}=K_I(u_x^j) S_m^*\circ K_I\circ\frac{\sigma_{a,m}}{K_I(u_x^j)}$ and $\Pi_\xi^{j*}=u_x^j S_x^*\circ K_I\circ\frac{\sigma_{a,x}^\eta}{K_I(u_x^j)}+ u_x^j S_x^* \circ K_I\circ\eta\sigma_{a,xf} S_m^*\circ K_I\circ\frac{\sigma_{a,m}}{K_I(u_x^j)}$ respectively. Here $S_z^*$ is the adjoint of $S_z$ ($z\in\{x,m\}$) that is defined as the solution operator of the adjoint transport equation with coefficients $\sigma_{a,z}$, $\sigma_{s,z}$ and vacuum boundary condition, i.e. $w_z=S_z^*(f)$ with $w_z$ the solution to:
\begin{equation*}
-\bv\cdot\nabla w_z + (\sigma_{a,z}+\sigma_{s,z}) w_x-\sigma_{s,z}K_\Theta(w_z)=f,\ \ \mbox{in}\ X, \qquad w_z=0\ \ \mbox{on}\ \Gamma_+ .
\end{equation*}
Therefore, at iteration $k$ of the Landweber algorithm, we solve $J$ forward transport systems and then $J$ adjoint transport systems to apply the operator $\bPi^*\bPi$ to the vector $(\zeta_k, \xi_k)^t$.

\subsection{Iterative reconstruction for the nonlinear case}
\label{SUBSEC:Nonl Rec}

For the simultaneous reconstruction of $\eta$ and $\sigma_{a,xf}$ in the general nonlinear case, we do not have any theoretical results on uniqueness and stability currently. Nor do we have more explicit reconstruction methods. We rely mostly on general computational optimization techniques to solve the inverse problem. More precisely, we search for solutions to the inverse problem by minimizing the objective functional:
\begin{equation}\label{EQ:OBJ}
	\Phi(\eta,\sigma_{a,xf}) 
		\equiv \dfrac{1}{2}\sum_{j=1}^{J}\int_{\Omega} \Big\{\Xi \big[\sigma_{a,x}^\eta K_I(u_x^j)+\sigma_{a,m}K_I(u_m^j) \big]-H_j\Big\}^2 d\bx+\beta R(\eta,\sigma_{a,xf})
\end{equation}
where the regularization functional is taken as $R(\eta,\sigma_{a,xf})=\frac{1}{2}(\|\nabla\eta\|_{[L^2(\Omega)]^d}^2+\|\nabla\sigma_{a,xf}\|_{[L^2(\Omega)]^d}^2)$.

Following the result in Proposition~\ref{PROP:Data QfPAT Frechet} and the chain rule, we can obtain the following result straightforwardly.
\begin{corollary}\label{COR:Deri Phi}
	The functional $\Phi(\eta,\sigma_{a,xf})$, viewed as the map: $\Phi: W_2^1(\Omega)\times W_2^1(\Omega) \mapsto \bbR_+$ is Fr\'echet differentiable at any $(\eta, \sigma_{a,xf})\in W_2^1(\Omega)\times W_2^1(\Omega) \cap\cA$. The partial derivatives in the direction $\delta\eta$ (such that $(\eta+\delta\eta,\sigma_{a,xf})\in\cA$) and the direction $\delta\sigma_{a,xf}$ (such that $(\eta,\sigma_{a,xf}+\delta\sigma_{a,xf})\in\cA$) are given respectively as
\begin{eqnarray}
	\Phi_\eta'[\eta,\sigma_{a,xf}](\delta\eta) & = & \int_{\Omega}\Big\{\sum_{j=1}^{J}z_j \Xi\big[-\delta\eta \sigma_{a,xf} K_I(u_x^j)+\sigma_{a,m}K_I(w_m^j)\big]+\beta\nabla\delta\eta\cdot\nabla \eta \Big\} d\bx,\\
	\Phi_\sigma'[\eta,\sigma_{a,xf}](\delta\sigma_{a,xf}) &= & \int_\Omega \sum_{j=1}^{J}z_j \Xi\big[\delta\sigma_{a,xf} (1-\eta)K_I(u_x^j)+\sigma_{a,x}^\eta K_I(v_x)+\sigma_{a,m}K_I(v_m)\big] d\bx\\
	& + & \beta \int_\Omega \nabla\delta\sigma_{a,xf}\cdot\nabla\sigma_{a,xf} d\bx,
\end{eqnarray}
where the residual $z_j=\Xi\big[\sigma_{a,x}^\eta K_I(u_x^j)+\sigma_{a,m} K_I(u_m^j) \big]-H_j$, $w_m$ is the unique solution to~\eqref{EQ:ERT QfPAT Linearized Eta}, and $(v_x, v_m)$ is the unique solution to~\eqref{EQ:ERT QfPAT Linearized Sigma}.
\end{corollary}

We can therefore employ gradient-based minimization techniques to minimize the functional~\eqref{EQ:OBJ}. Here we use the limited memory version of the BFGS quasi-Newton method that we implemented in~\cite{ReBaHi-SIAM06}. This method requires only the gradients of the objective functional which we derived in Corollary~\ref{COR:Deri Phi}. To simplify the computation of these gradients numerically, we apply the adjoint state technique. We denote by $(q_x^j,q_m^j)$ the unique solution to the following adjoint transport system:
\begin{equation}\label{EQ:ERT QfPAT Adj}
\begin{array}{rcll}
	-\bv\cdot\nabla q_x^j+\sigma_{t,x} q_x^j&=& \sigma_{s,x} K_\Theta(q_x^j) + \Xi \sigma_{a,x}^\eta z_j+\eta \sigma_{a,xf} K_I(q_m^j), &\mbox{in}\ \ X\\
	-\bv\cdot\nabla q_m^j+\sigma_{t,m} q_m^j &=& \sigma_{s,m} K_\Theta(q_m^j)+\Xi \sigma_{a,m} z_j, &\mbox{in}\ \ X\\
    q_x^j(\bx,\bv) = 0,&&  
    q_m^j(\bx,\bv) = 0 & \mbox{on}\ \ \Gamma_+ .
	\end{array}	
\end{equation}
It is then straightforward to show that
\begin{eqnarray}
\label{EQ:Phi Grad E} 
\Phi_\eta'[\eta,\sigma_{a,xf}](\delta\eta) & = & \int_{\Omega} \Big\{\dsum_{j=1}^{J}\delta\eta \sigma_{a,xf} K_I(u_x^j) \big[-\Xi z_j  + K_I(q_m^j)\big]+\beta\nabla\delta\eta\cdot\nabla \eta \Big\}d\bx,\\
\nonumber 
\Phi_\sigma'[\eta,\sigma_{a,xf}](\delta\sigma_{a,xf}) & = & \int_\Omega \dsum_{j=1}^{J} \delta\sigma_{a,xf}K_I(u_x^j)\big[\Xi(1-\eta)z_j+\eta K_I(q_m)-K_I(q_x)\big] d\bx\\
\label{EQ:Phi Grad S}
& + & \beta\int_\Omega \nabla\delta\sigma_{a,xf}\cdot\nabla\sigma_{a,xf} d\bx.
\end{eqnarray}
Therefore, to compute gradients of the $\Phi$ at $(\eta, \sigma_{a,xf})$, we only need to solve a set of $J$ forward transport systems~\eqref{EQ:ERT QfPAT} and a set of $J$ adjoint transport systems~\eqref{EQ:ERT QfPAT Adj}. We can then evaluate the gradients in any given direction $(\delta\eta, \delta\sigma_{a,xf})$ according to~\eqref{EQ:Phi Grad E} and~\eqref{EQ:Phi Grad S}.

It is obvious that this optimization-based nonlinear reconstruction method can be used also to reconstruct a single coefficient. To only reconstruct $\eta$, we only need to set the gradient with respect to $\sigma_{a,xf}$ to zero and vice versa.

\section{Numerical Experiments}
\label{SEC:Num}

We now present some numerical reconstructions using synthetic interior data. We restrict ourselves to two-dimensional settings only to simplify the computation. 

The spatial domain of the reconstruction is the square $\Omega=(-1,1)\times (-1,1)$. All the transport equations in $\Omega\times\bbS^1$ are discretized angularly with the discrete ordinate method and spatially with a first-order finite element method on triangular meshes. In all the simulations in this section, reconstructions are performed on a finite element mesh consisting of about $2000$ triangles and a discrete ordinate set with $64$ directions. For the absorption and scattering coefficients that are known, we take
\begin{align}
\label{EQ:uaxi}	\sigma_{a,xi}&=\sigma_{a,m}= \sigma_{a}^b\left(2 -(\lfloor 2x \rfloor +\lfloor 2y \rfloor\mod{2})\right),\\
\label{EQ:usx}	\sigma_{s,x}&=\sigma_{s,m}= \sigma_{s}^b\left(1+(\lfloor 2x \rfloor +\lfloor 2y \rfloor\mod{2})\right),
\end{align}
where $\lfloor\cdot\rfloor$ represents the floor operation, $\sigma_{a}^b$ and $\sigma_{s}^b$ are respectively the base level absorption and scattering coefficients. In all the cases below, we set $\sigma_{a}^b=0.1$. The value of $\sigma_{s}^b$ varies from case to case and will be given below; see Fig.~\ref{FIG:Sim Setup} (i) and (ii) for plots of the two coefficients. The scattering kernel $\Theta$ is set to be the Henyey-Greenstein phase function~\cite{Arridge-IP99,HeGr-AJ41,WeVa-Book95} which depends only on the product $\bv\cdot\bv'$. 

To generate synthetic data for the nonlinear inversions, we solve the transport system~\eqref{EQ:ERT QfPAT} with true quantum efficiency $\eta$ and fluorescent absorption coefficient $\sigma_{a,xf}$ and compute $H$ according to~\eqref{EQ:Data QfPAT}. To generate synthetic data for linearized inversions, for instance in Experiment 3 below, we use directly the linearized data models, for instance~\eqref{EQ:Data QfPAT Pert}, with the true coefficient perturbations. This way, we can exclude the linearization error from the data used in the inversion. We do this since our main aim is to test the performance of the reconstruction algorithms, not to check the accuracy of the linearizations. To mimic noisy measurements, we add additive random noise to the synthetic data by multiplying each datum point by $(1+\gamma \times 10^{-2}\texttt{normrnd})$ with \texttt{normrnd} a standard Gaussian random variable and $\gamma$ a number representing the noise level in percentage. When $\gamma=0$, we say the data are noise-free.

To measure the quality of the reconstruction, we use the relative $L^2$ error. This error is defined as the $L^2$ norm of the difference between the reconstructed coefficient and the true coefficient, divided by the $L^2$ norm of the true coefficient and then multiplied by $100$. 

We performed numerical simulations on the reconstructions of many different coefficients pairs $(\eta, \sigma_{a,xf})$. The qualities of the the reconstructions are very similar. To avoid repetition, we will present only reconstructions for a typical coefficient pair we show in (iii)-(iv) of Fig.~\ref{FIG:Sim Setup}.
\begin{figure}[hbtp] 
\centering
\includegraphics[width=0.24\textwidth,height=0.22\textwidth]{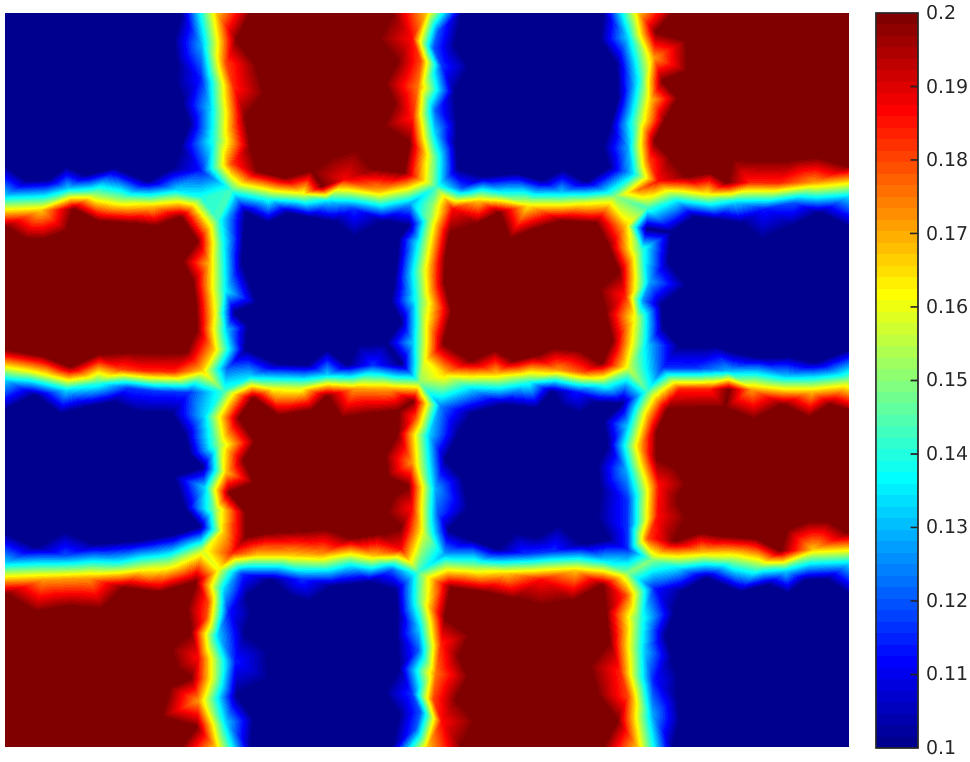}
\includegraphics[width=0.24\textwidth,height=0.22\textwidth]{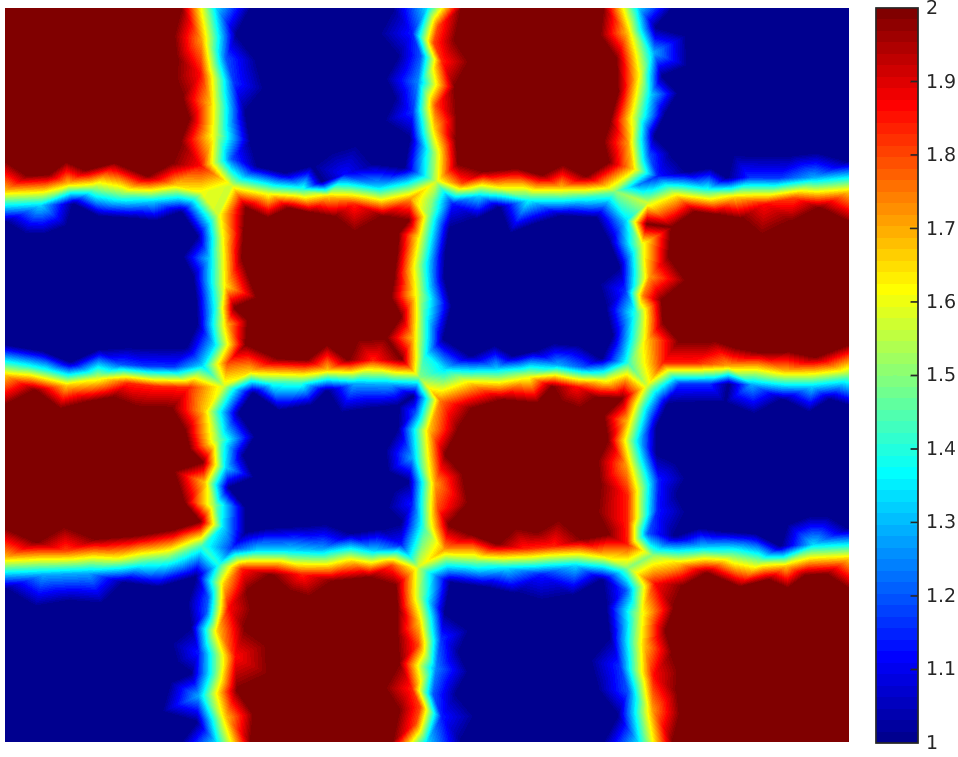} 
\includegraphics[width=0.24\textwidth,height=0.22\textwidth]{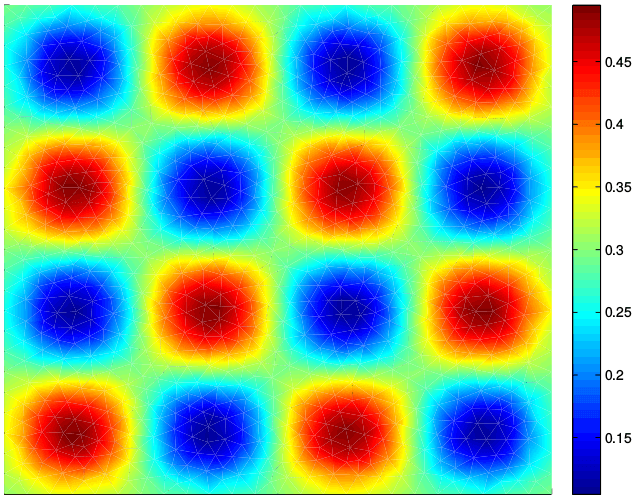}
\includegraphics[width=0.24\textwidth,height=0.22\textwidth]{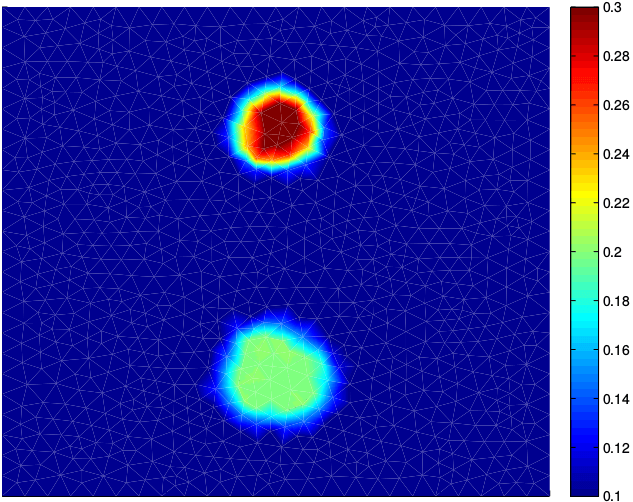} 
\caption{From left to right are: (i) the absorption coefficient $\sigma_{a,xi}=\sigma_{a,m}$ defined in~\eqref{EQ:uaxi} with $\sigma_a^b=0.1$, (ii) the scattering coefficient $\sigma_{s,x}=\sigma_{s,m}$ defined in~\eqref{EQ:usx} with $\sigma_s^b=2.0$, (iii) the true quantum efficiency $\eta$ to be reconstructed  in the numerical experiments, and (iv) the true fluorescence absorption coefficient $\sigma_{a,xf}$ to be reconstructed.} 
\label{FIG:Sim Setup}
\end{figure}

\paragraph{Experiment 1.} In the first set of numerical studies, we consider the reconstruction of the quantum efficiency $\eta$ assuming that the fluorescent absorption coefficient $\sigma_{a,xf}$ is known. We use the Reconstruction Algorithm I presented in Section~\ref{SUBSEC:Eta}. We first perform numerical experiments in isotropic medium with two different strengths of scattering coefficients. We show in Fig.~\ref{fig:qne_us1} the reconstructions of $\eta$ under base scattering $\sigma_{s}^b=1.0$. Shown from left to right are respectively the $\eta$ reconstructed using data with noise level $\gamma=0$, $2$, $5$ and $10$ respectively. The relative $L^2$ errors in the reconstructions are respectively $0.01\%$, $14.24\%$, $35.59\%$ and $71.18\%$. We repeat the simulations for a medium with stronger (but still isotropic) scattering ($\sigma_{s}^b=9.0$). The results are shown in Fig.~\ref{fig:qne_us9}. The relative $L^2$ errors in this case are $1.04\%$, $14.84\%$, $37.02\%$ and $74.02\%$ respectively. If we compare the results in Fig.~\ref{fig:qne_us1} and those in Fig.~\ref{fig:qne_us9}, we see that the quality of the reconstructions are almost independent of the scattering strength. This is what we observed in our numerical experiments in other cases as well.
\begin{figure}[hbtp] 
\centering 
\includegraphics[width=0.24\textwidth,height=0.22\textwidth]{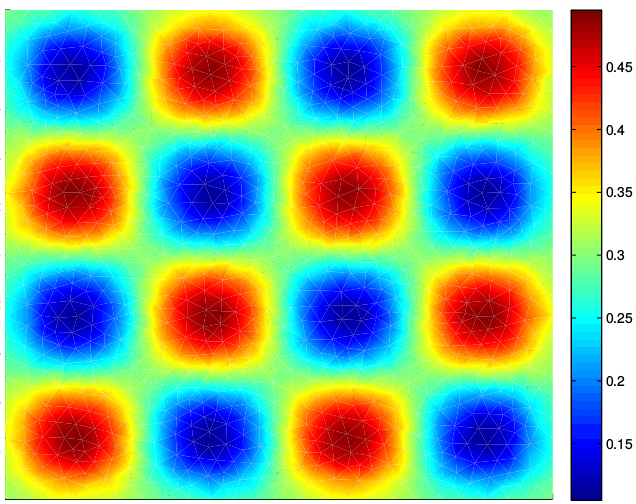} 
\includegraphics[width=0.24\textwidth,height=0.22\textwidth]{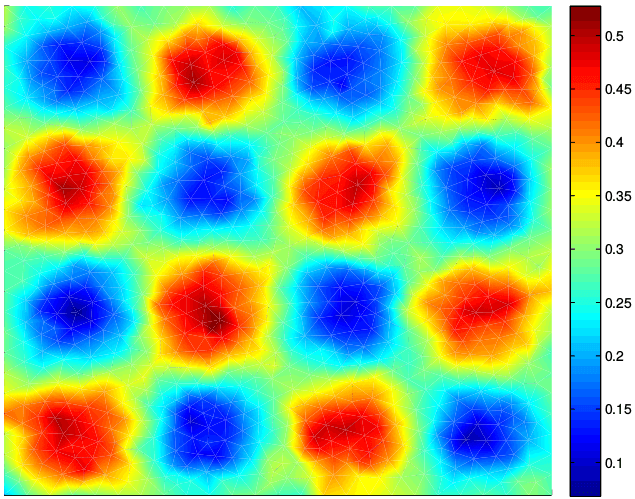} 
\includegraphics[width=0.24\textwidth,height=0.22\textwidth]{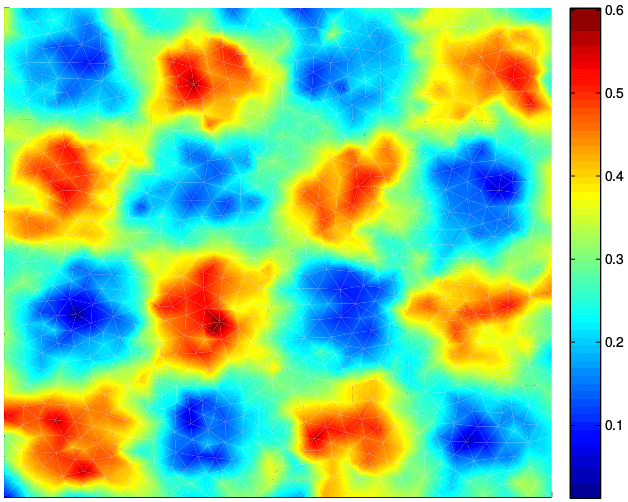} 
\includegraphics[width=0.24\textwidth,height=0.22\textwidth]{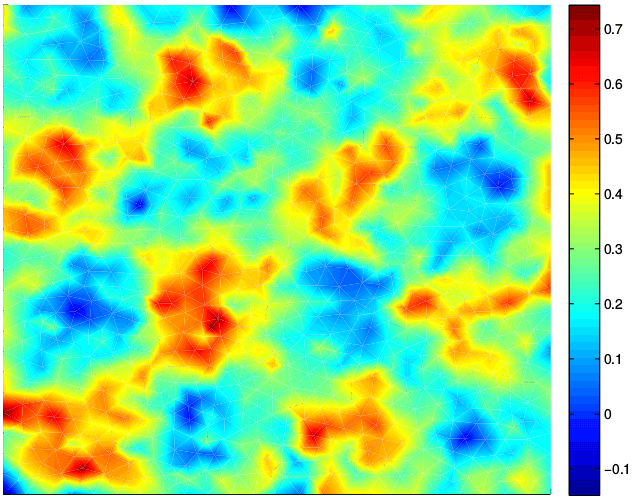} 
\caption{The quantum efficiency $\eta$ reconstructed with different types of data. The noise levels in the data used for the reconstructions, from left to right are $\gamma=0$, $2$, $5$ and $10$ respectively. The base scattering strength is $\sigma_{s}^b=1.0$.} 
\label{fig:qne_us1}
\end{figure}
\begin{figure}[hbtp] 
\centering 
\includegraphics[width=0.24\textwidth,height=0.22\textwidth]{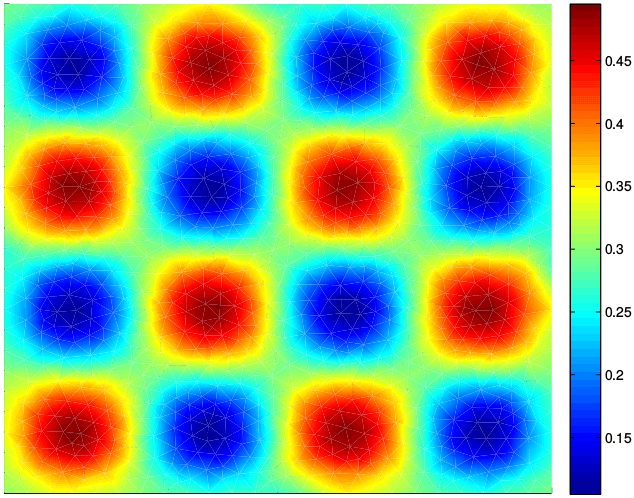} 
\includegraphics[width=0.24\textwidth,height=0.22\textwidth]{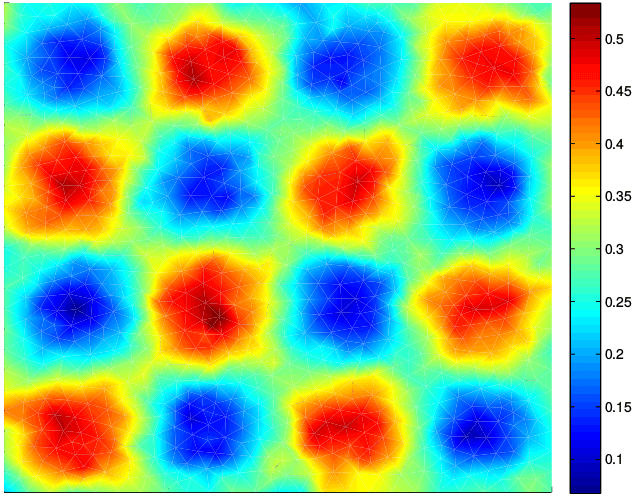} 
\includegraphics[width=0.24\textwidth,height=0.22\textwidth]{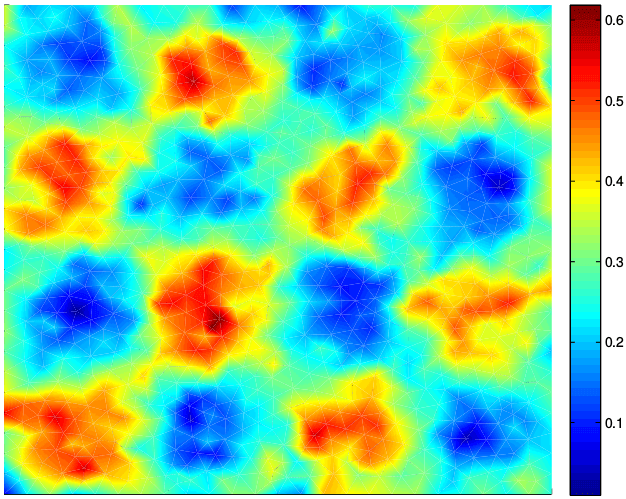} 
\includegraphics[width=0.24\textwidth,height=0.22\textwidth]{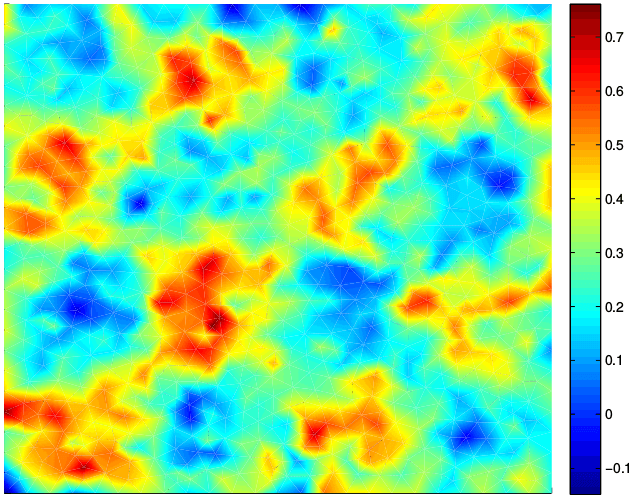} 
\caption{Same as in Fig.~\ref{fig:qne_us1} but with base scattering strength $\sigma_{s}^b=9.0$.}
\label{fig:qne_us9}
\end{figure}

\paragraph{Experiment 2.} In the second set of numerical studies, we consider the reconstruction of the fluorescent absorption coefficient $\sigma_{a,xf}$ assuming that the quantum efficiency $\eta$ is known. Currently, we do not have a well-established method to construct illuminations sources such that the condition $u_x=K_I(u_x)$ is satisfied for the transport solution, besides in non-scattering media. We therefore can not use directly the Reconstruction Algorithm II as we commented before. Instead, we use the nonlinear reconstruction algorithm in Section~\ref{SUBSEC:Nonl Rec}. We show in Fig.~\ref{fig:uxaf_us1} the reconstructions of $\sigma_{a,xf}$ in an isotropic medium with base scattering strength $\sigma_{s}^b=1.0$. Shown from left to right are respectively the reconstructions using data with noise levels $\gamma=0$, $2$, $5$ and $10$. The relative $L^2$ errors in the four reconstructions are $0.01\%$, $6.42\%$,$16.06\%$ and $32.12\%$ respectively. In Fig.~\ref{fig:uxaf_us9}, we show the same reconstructions in an anisotropic scattering medium with base scattering strength $\sigma_{s}^b=9.0$ and anisotropic factor $0.9$. The relative $L^2$ errors are $0.02\%$,$6.70\%$,$16.74\%$ and $33.42\%$, respectively. We again observed that the reconstructions are of good quality with data contains reasonably low level of random noise.
\begin{figure}[hbtp] 
\centering
\includegraphics[width=0.24\textwidth,height=0.22\textwidth]{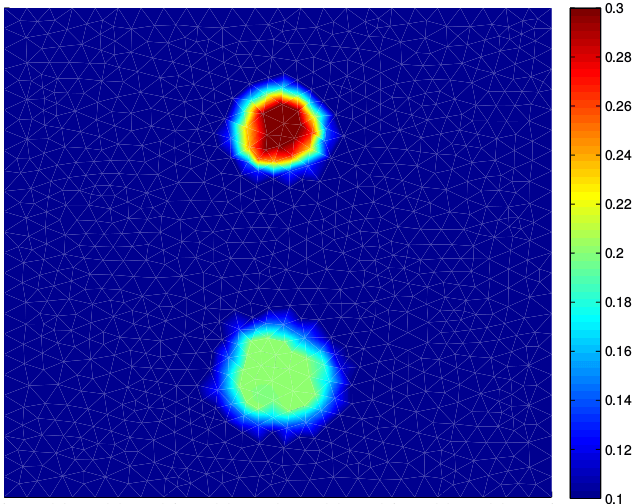} 
\includegraphics[width=0.24\textwidth,height=0.22\textwidth]{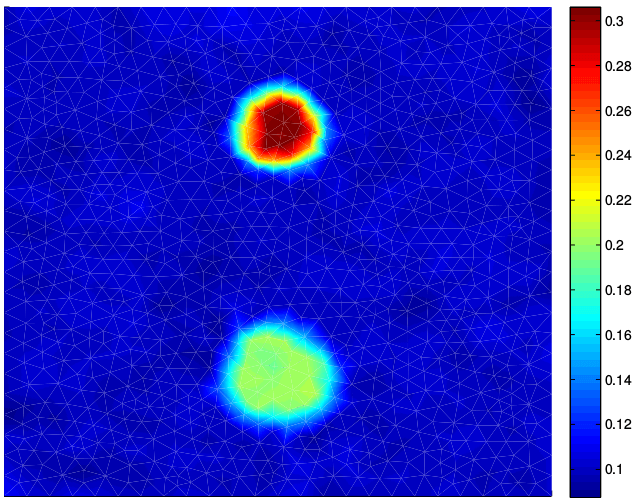} 
\includegraphics[width=0.24\textwidth,height=0.22\textwidth]{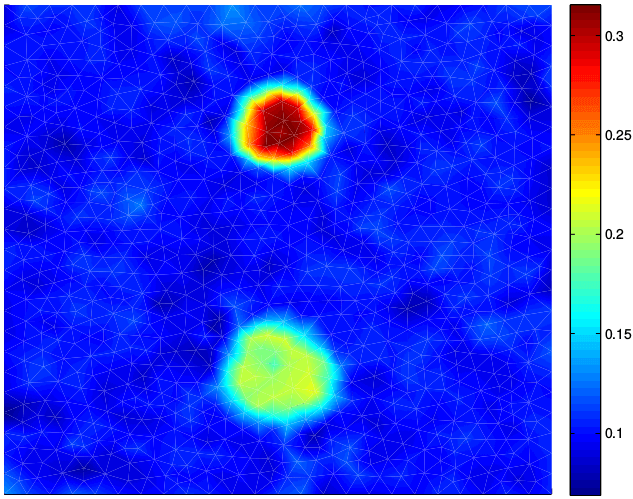} 
\includegraphics[width=0.24\textwidth,height=0.22\textwidth]{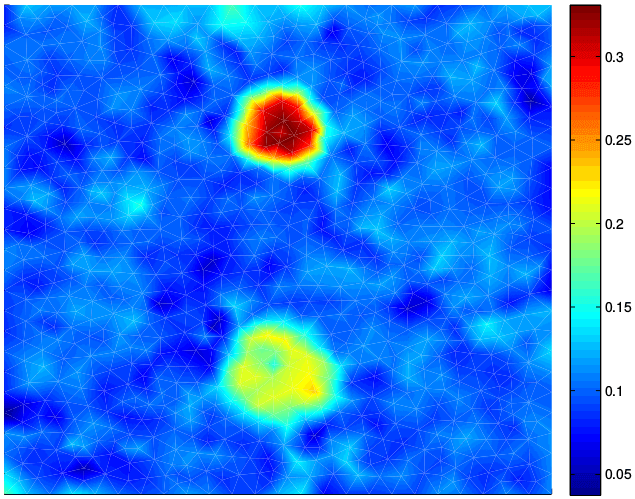} 
\caption{The fluorescence absorption coefficient $\sigma_{a,xf}$ reconstructed with different types of data. The noise level in the data used for the reconstructions, from left to right are: $\gamma=0$ (noise-free), $\gamma=2$, $\gamma=5$, and $\gamma=10$. The base scattering strength is $\sigma_{s}^b=1.0$.} 
\label{fig:uxaf_us1}
\end{figure}
\begin{figure}[hbtp] 
\centering
\includegraphics[width=0.24\textwidth,height=0.22\textwidth]{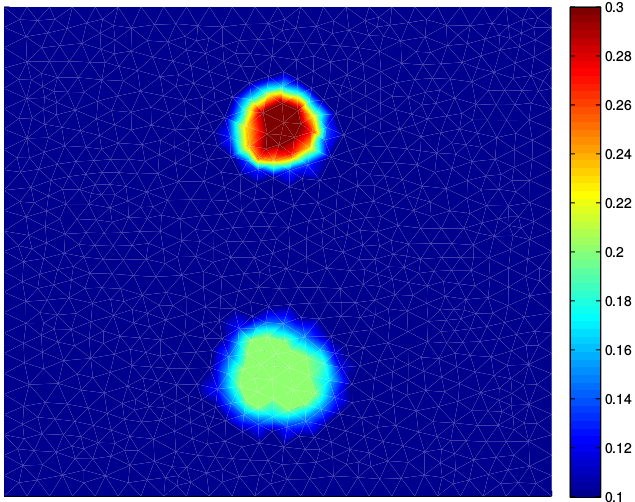} 
\includegraphics[width=0.24\textwidth,height=0.22\textwidth]{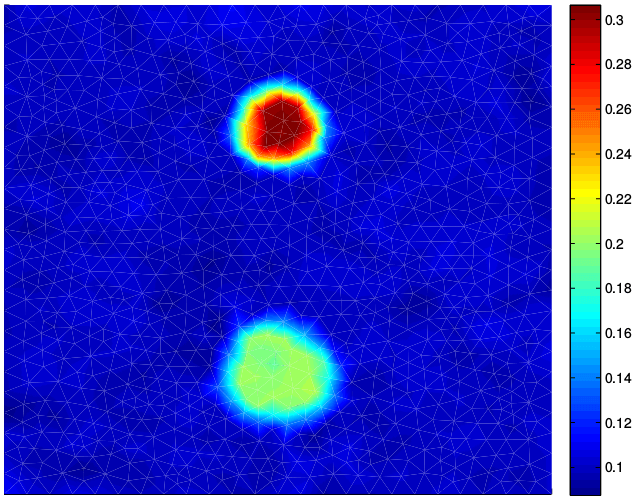} 
\includegraphics[width=0.24\textwidth,height=0.22\textwidth]{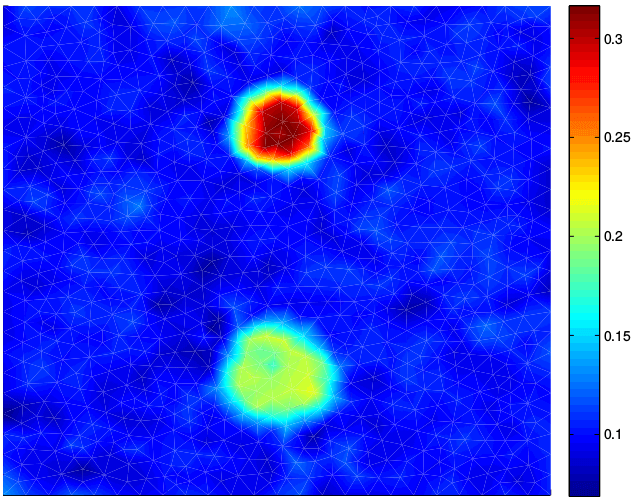} 
\includegraphics[width=0.24\textwidth,height=0.22\textwidth]{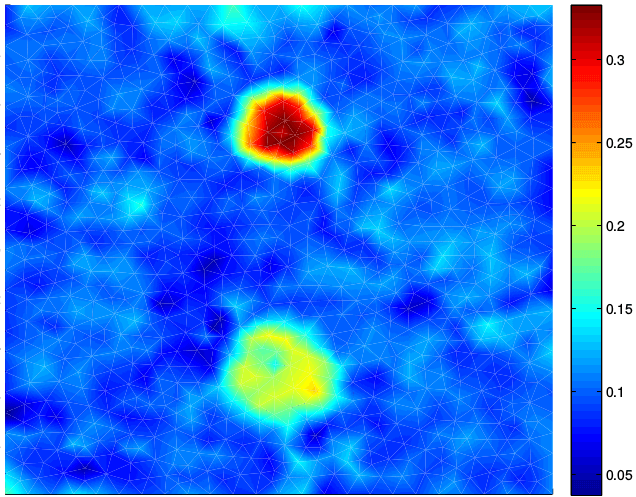} 
\caption{Same as in Fig.~\ref{fig:uxaf_us1} but in a medium of anisotropic scattering with base scattering strength $\sigma_{s}^b=9.0$ and anisotropic factor $0.9$.} 
\label{fig:uxaf_us9}
\end{figure}

\paragraph{Experiment 3.} In the third set of numerical simulations, we study the simultaneous reconstruction of the coefficients $\eta$ and $\sigma_{a,xf}$ in the linearized setting described in Section~\ref{SUBSEC:Lin Eta 0} using the Reconstruction Algorithm III. The synthetic perturbed data are generated using directly the linearized model~\eqref{EQ:Data QfPAT Pert}, not the original nonlinear model. Our aim here is to test the stability of the reconstruction, not the accuracy of the linearization. We use data sets collected from four angularly-resolved illuminations supported respectively on the four sides of the boundaries of the domain, pointing toward the interior of the domain. The background scattering strength is $\sigma_{s}^b=1.0$ and the anisotropic factor is $0.5$. We linearize the problem around the background coefficients:
\[
	\eta^0=\dfrac{1}{|\Omega|}\dint_\Omega\eta(\bx) d\bx\ \ \mbox{and}\ \ 
\sigma_{a,xf}^0=\dfrac{1}{|\Omega|}\dint_\Omega\sigma_{a,xf}(\bx) d\bx .
\]
The reconstructions, after adding back the background, are shown in Fig.~\ref{fig:uaxf_qne_us0}. The relative $L^2$ error in the reconstructions using data with noise level $\gamma=0$, $\gamma=2$, $\gamma=5$ and $\gamma=10$ are respectively $(0.00\%, 0.00\%)$, $(14.65\%, 7.45\%)$, $(37.28\%, 18.77\%)$ and $(75.80\%, 39.04\%)$ respectively. In all reconstructions, we applied the Tikhonov regularization with a small regularization strength that we select by trial and errors. We hope to develop more systematical strategy on regularization in the future. 
\begin{figure}[hbtp] 
\centering 
\includegraphics[width=0.24\textwidth,height=0.22\textwidth]{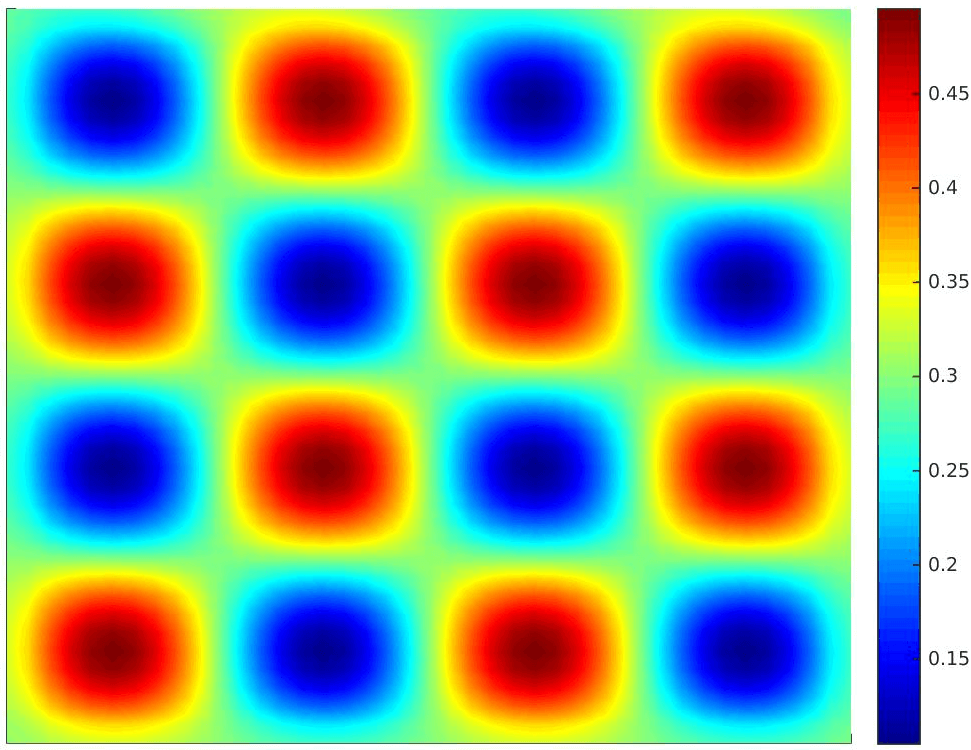} 
\includegraphics[width=0.24\textwidth,height=0.22\textwidth]{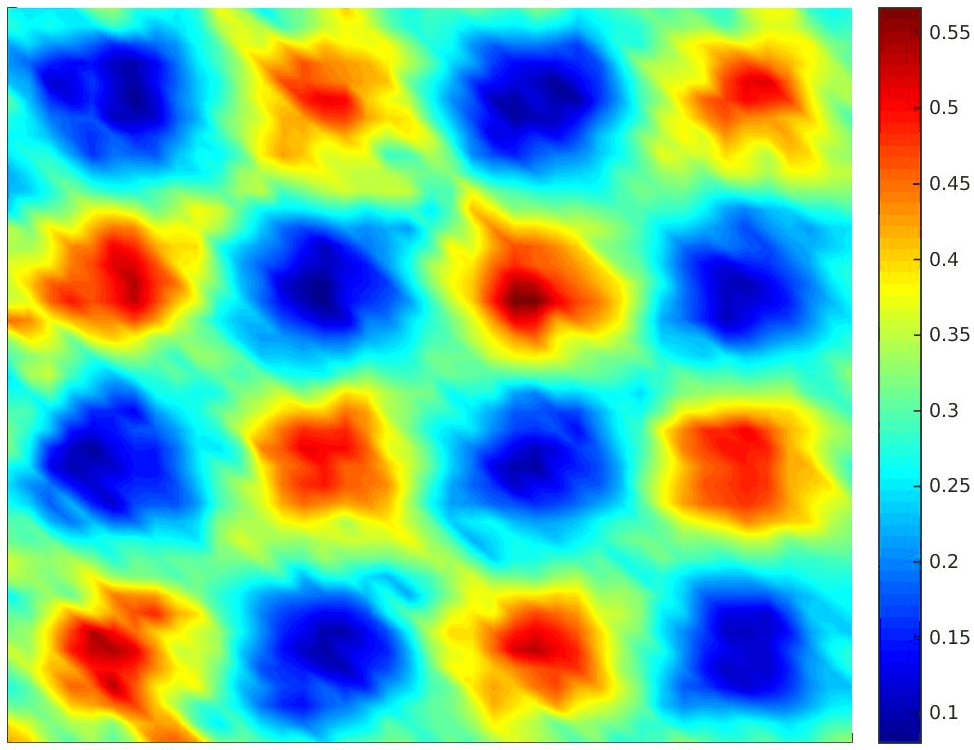} 
\includegraphics[width=0.24\textwidth,height=0.22\textwidth]{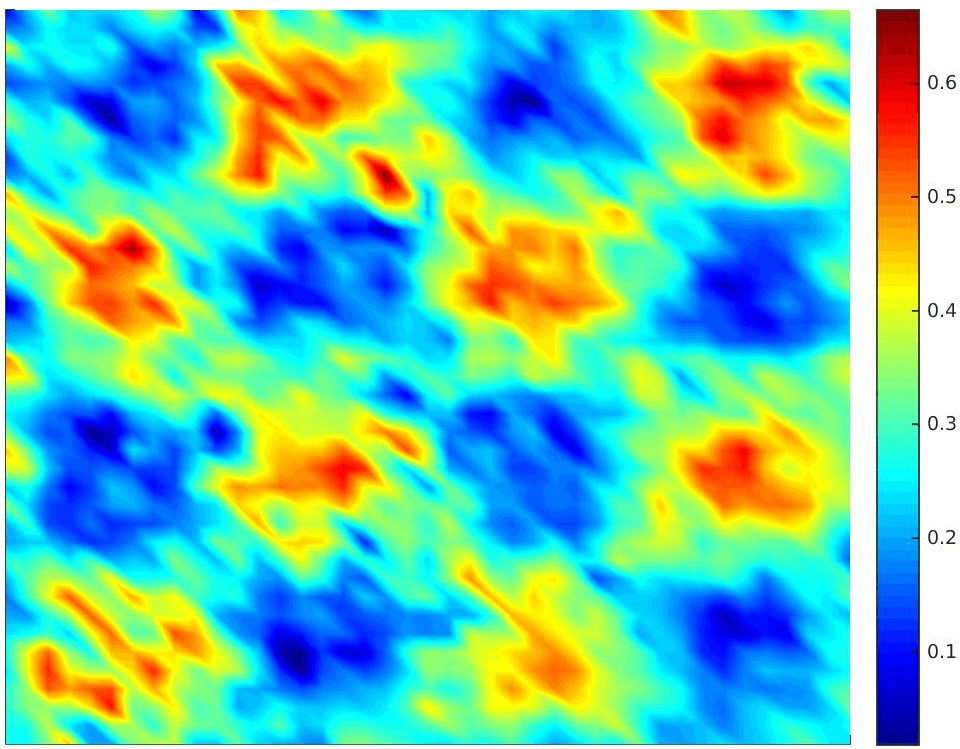} 
\includegraphics[width=0.24\textwidth,height=0.22\textwidth]{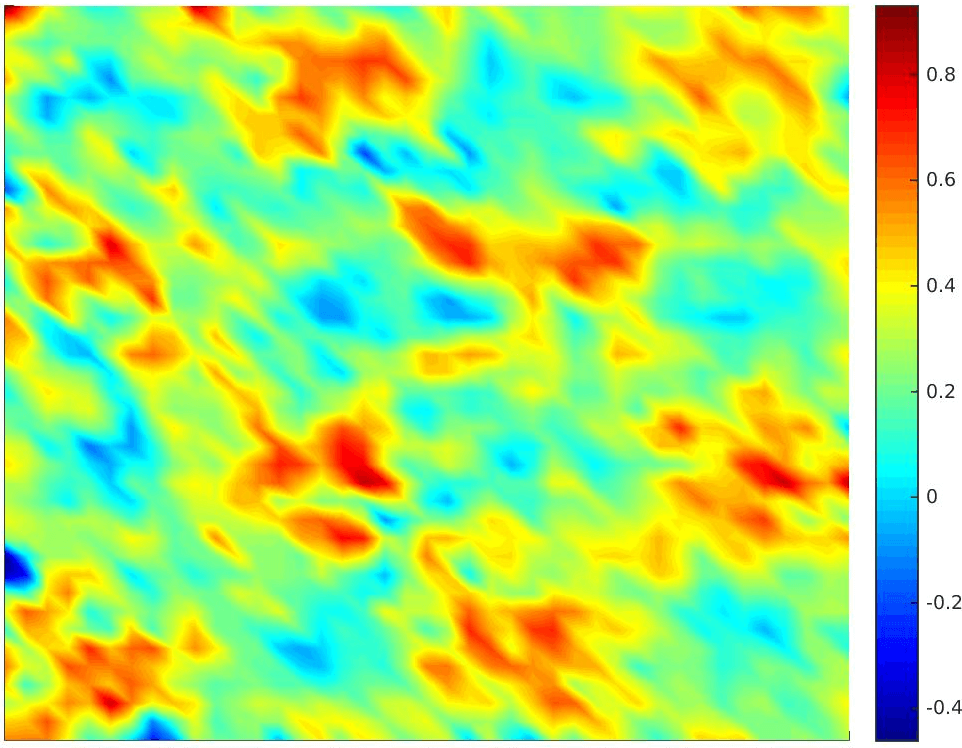}\\ 
\includegraphics[width=0.24\textwidth,height=0.22\textwidth]{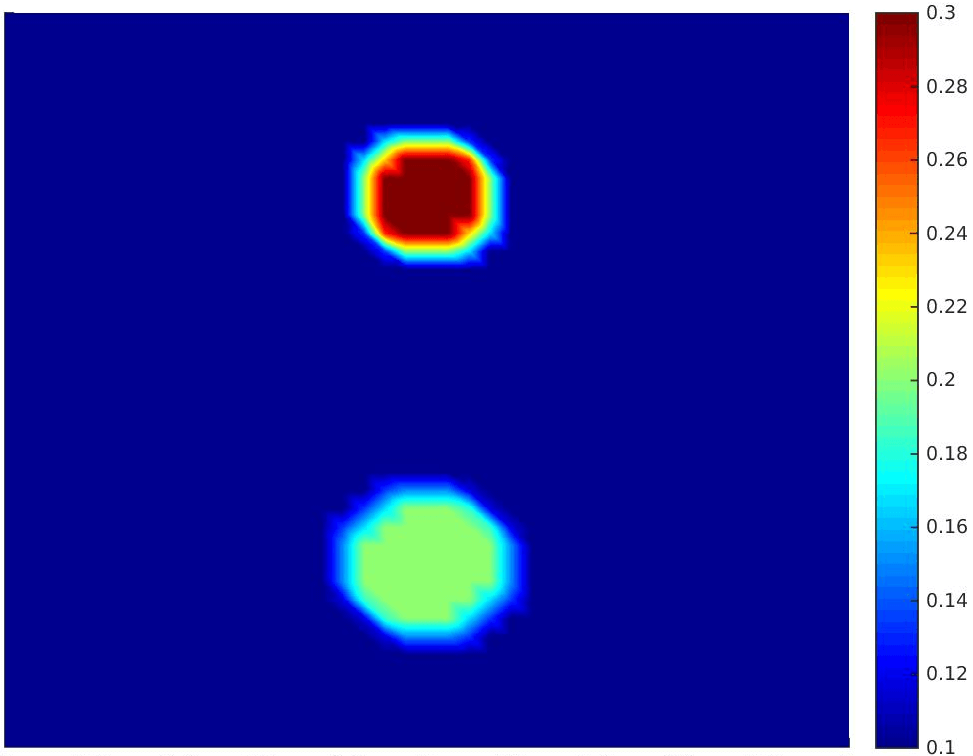} 
\includegraphics[width=0.24\textwidth,height=0.22\textwidth]{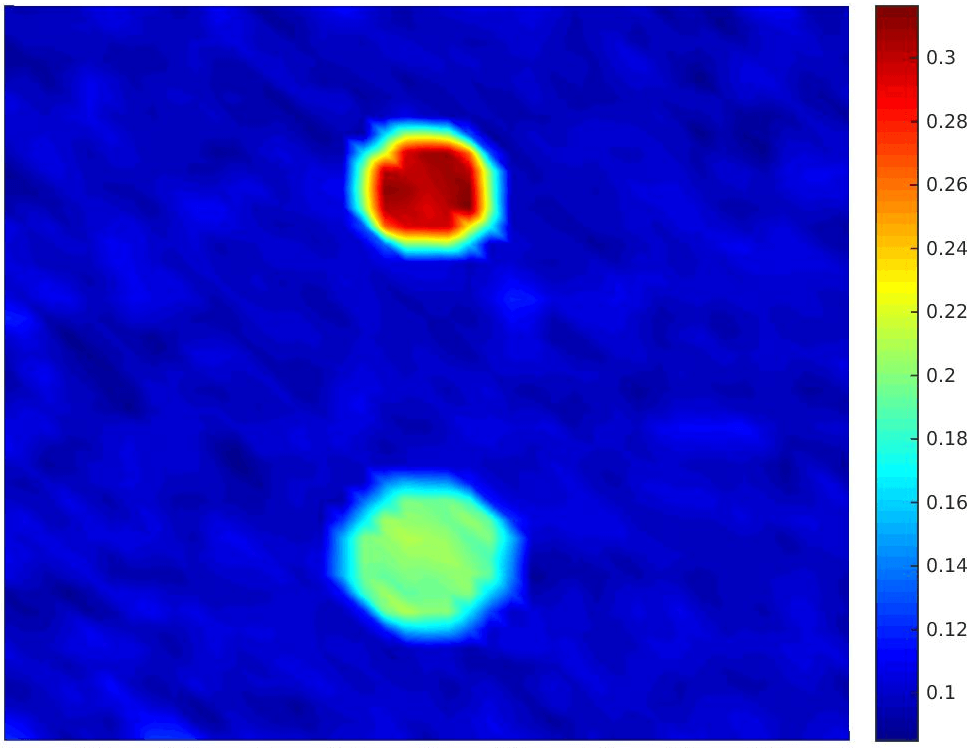} 
\includegraphics[width=0.24\textwidth,height=0.22\textwidth]{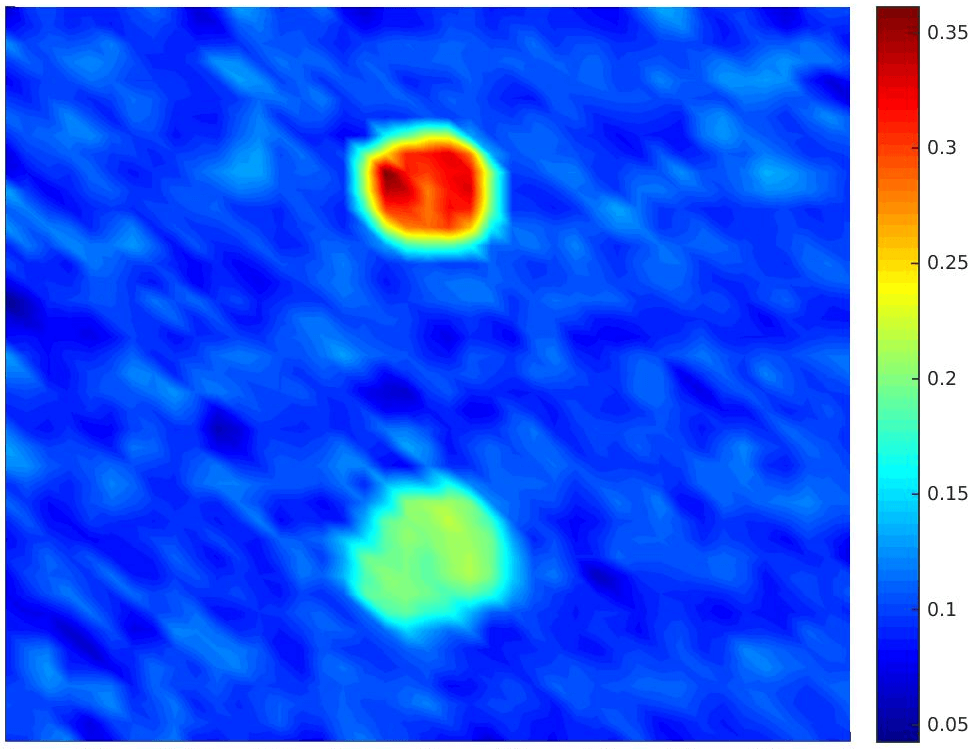} 
\includegraphics[width=0.24\textwidth,height=0.22\textwidth]{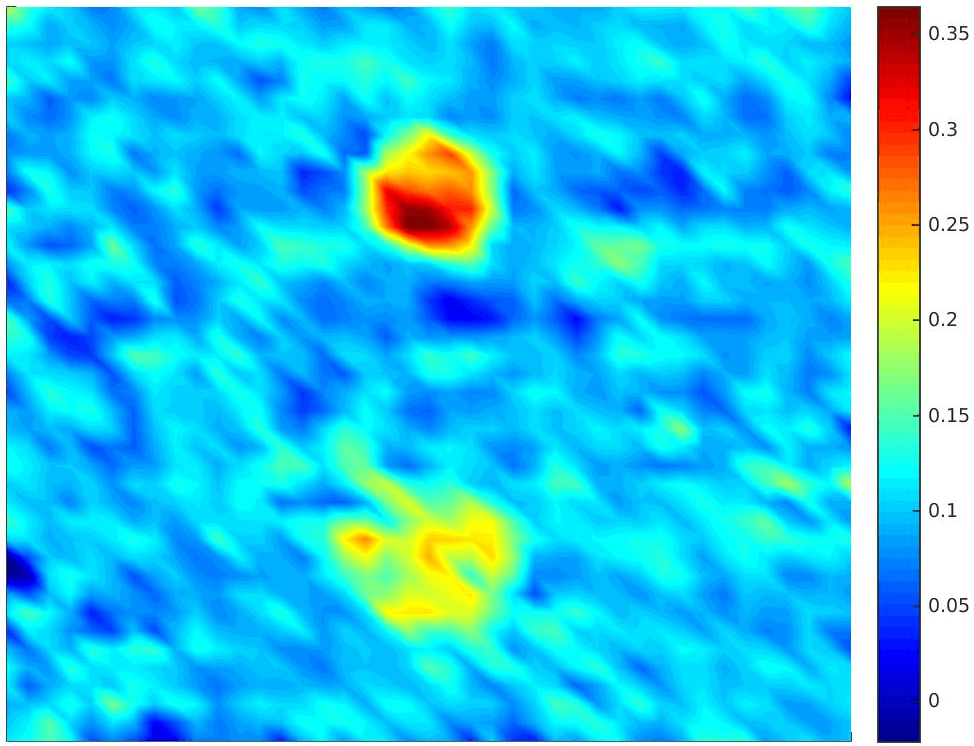} 
\caption{Simultaneous reconstructions of the coefficient pair $(\eta, \sigma_{a,xf})$ in the linearized setting with different types of data. The noise level in the data used for the reconstructions are (from left to right): $\gamma=0$, $2$, $5$ and $10$ respectively. The base scattering strength is $\sigma_{s}^b=1.0$.}
\label{fig:uaxf_qne_us0}
\end{figure}

\paragraph{Experiment 4.} The last set of numerical simulations are devoted to the simultaneous reconstructions of the coefficient pair $(\eta, \sigma_{a,xf})$ in the fully nonlinear setting. We use the optimization-based reconstruction algorithm developed in Section~\ref{SUBSEC:Nonl Rec}. Besides the fact that the synthetic data are now generated with the full transport model~\eqref{EQ:ERT QfPAT}, not the linearized model~\eqref{EQ:Data QfPAT Pert}, the setup (for instance the background coefficients and anisotropic factor etc) is the same as that in Experiment 3. We performed reconstructions with data containing various noise levels. When the noise level is too high, we have difficulties to find reasonable initial guesses to make the algorithm converge. We show in Fig.~\ref{FIG:Simult Nonl} reconstructions with data containing a small amount of noise, $\gamma=0$, $1$ and $2$ respectively, with the initial guess $(\eta^0, \sigma_{a,xf}^0)$ being the average of the true coefficients inside the domain. The relative $L^2$ error in the reconstructions are respectively $(16.40\%, 8.32\%)$, $(18.26\%, 9.17\%)$ and $(23.26\%, 19.30\%)$ respectively. We again impose weak Tikhonov regularizations in all the reconstructions with the regularization strengths selected by trial and error. Tuning various parameters in the algorithm could potentially improve the reconstructions results, but we did not pursue in that direction.
\begin{figure}[hbtp] 
\centering 
\includegraphics[width=0.255\textwidth,height=0.22\textwidth]{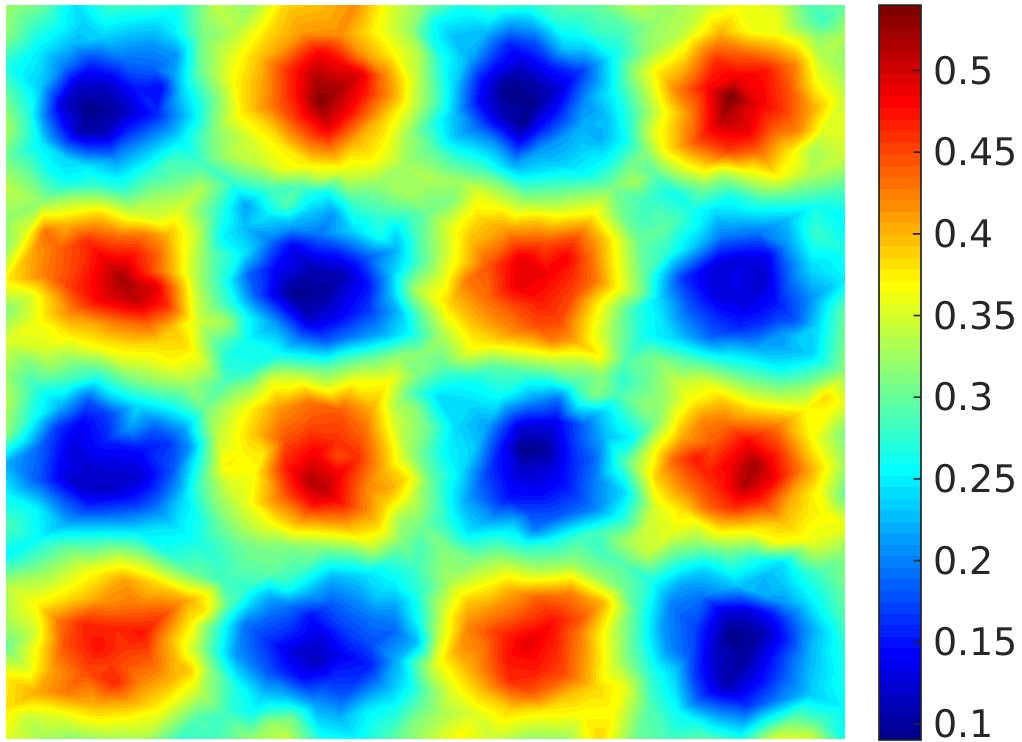} 
\includegraphics[width=0.255\textwidth,height=0.22\textwidth]{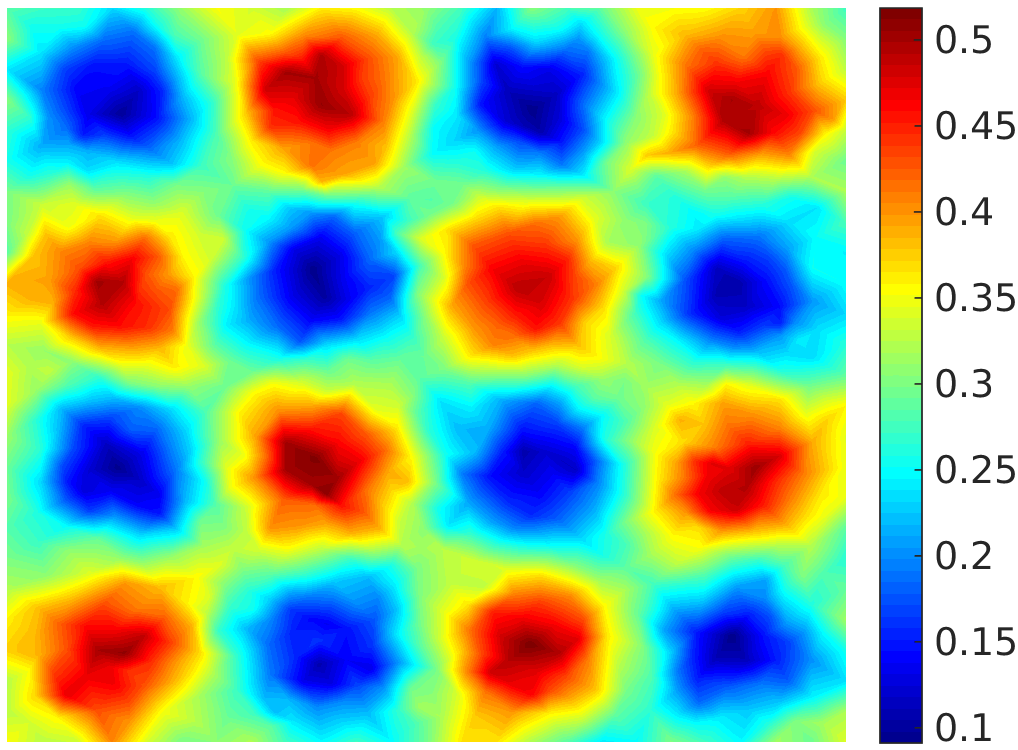}
\includegraphics[width=0.255\textwidth,height=0.22\textwidth]{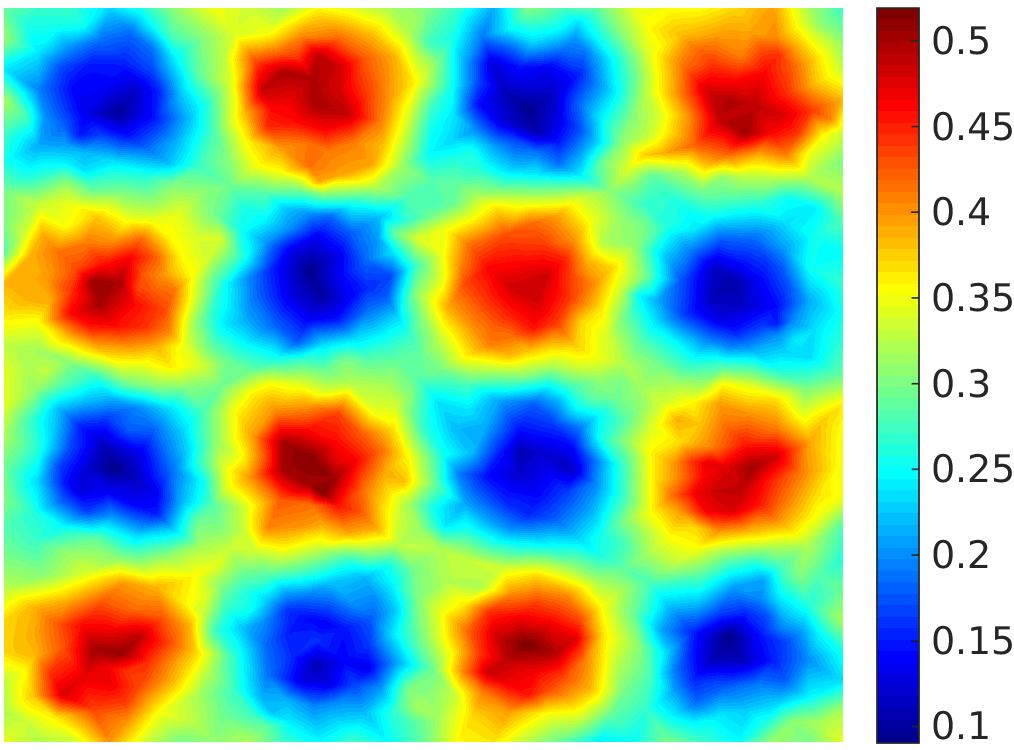}\\ 
\includegraphics[width=0.255\textwidth,height=0.22\textwidth]{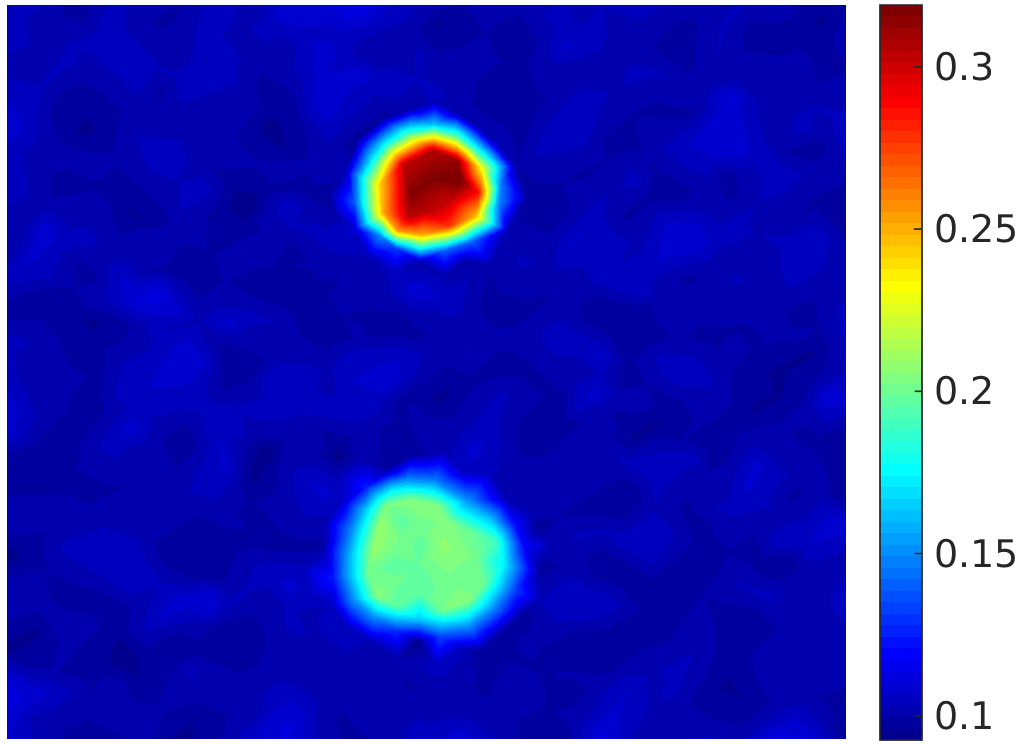} 
\includegraphics[width=0.255\textwidth,height=0.22\textwidth]{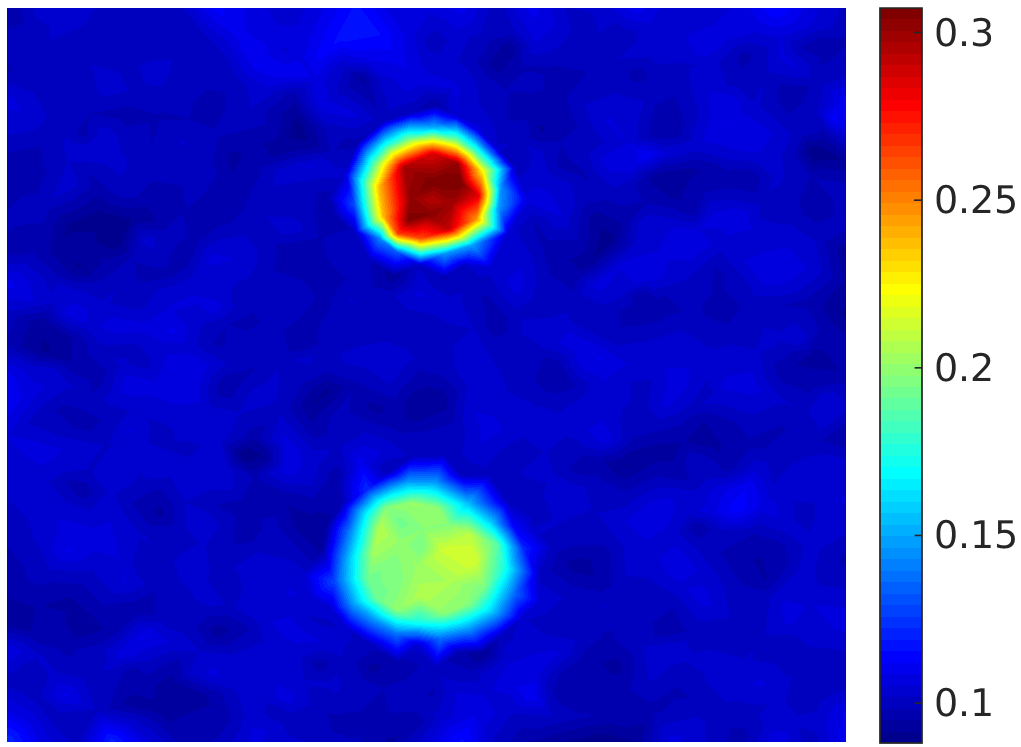} 
\includegraphics[width=0.255\textwidth,height=0.22\textwidth]{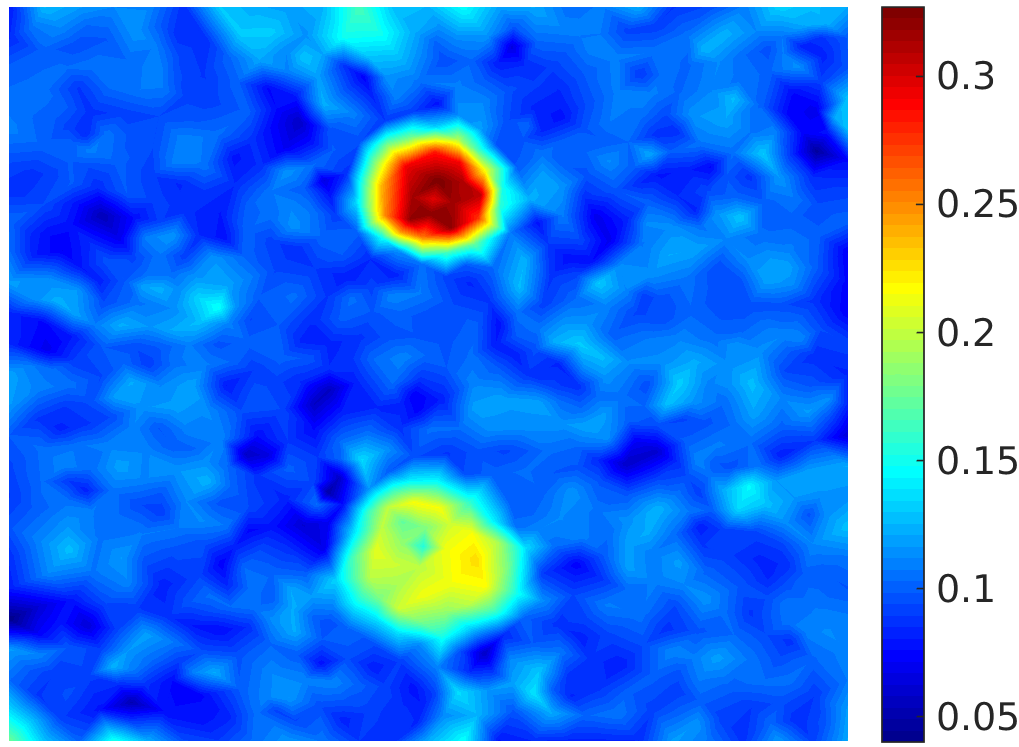} 
\caption{Simultaneous reconstruction of the coefficient pair $(\eta, \sigma_{a,xf})$ in the nonlinear setting with different types of data. The noise level in the data used for the reconstructions, from left to right, are respectively $\gamma=0$, $1$ and $2$.} 
\label{FIG:Simult Nonl}
\end{figure}

\section{Concluding Remarks}
\label{SEC:Concl}

We studied in this work a few inverse problems in quantitative fluorescence photoacoustic tomography in the radiative transport regime. We derived some uniqueness and stability results on the reconstruction of the fluorescence absorption coefficient and the quantum efficiency of the medium. In some cases, we were also able to develop efficient numerical reconstruction algorithms. These results complement the results in~\cite{ReZh-SIAM13} for the QfPAT problem in the diffusive regime. We showed numerical simulations based on synthetic data to support the mathematical analysis and demonstrate the performance of some of the reconstruction algorithms.

One important application of the results in this paper is in X-ray modulated fluorescence tomography (or X-ray luminescence tomography (XLT))~\cite{StCoWa-IPI15}. In XLT, X-rays, instead of NIR photons, are used to excite the molecular markers. The X-ray density $u_x$ and the generated NIR photon densities $u_m$ solve the coupled transport system~\eqref{EQ:ERT QfPAT} with the scattering term $K_\Theta(u_x)=0$ since X-rays travel in straight lines without being scattered. The theory and reconstruction methods we developed in this work remain valid in that case. In other words, we can recover stably the fluorescence absorption coefficient using data collected from one X-ray illumination. This would provide a useful alternative to the reconstruction method for XLT in~\cite{StCoWa-IPI15}.

Even though the QfPAT problem has been analyzed in detail in~\cite{ReZh-SIAM13} in the diffusive regime, the developments in this work are still useful in many settings. One well-known example is the application in optical imaging of small animals~\cite{Hielscher-COB05} where the diffusion model is not sufficiently accurate to describe the propagation of NIR photons inside the animals.

Our main research focus in near future is to analyze the uniqueness and stability properties of the simultaneous reconstruction problem, i.e. the problem of reconstructing the pair $(\eta, \sigma_{a,xf})$, in the fully nonlinear setting. This is an unsolved problem even in the diffusive regime~\cite{ReZh-SIAM13}, although numerical simulations we have so far suggested that uniqueness and stability both hold, at least in the regime where both coefficients are sufficiently large.

\section*{Acknowledgments}

We would like to thank the anonymous referees whose comments help us improve significantly the quality of this paper. This work is partially supported by the National Science Foundation through grant DMS-1321018, and the University of Texas at Austin through a Moncrief Grand Challenge Faculty Award.




\end{document}